\documentclass{article}

\usepackage{amsmath,amsfonts,bm}

\def\eqref#1{equation~(\ref{#1})}
\def\Eqref#1{Equation~(\ref{#1})}

\def\1{\bm{1}}

\def\rva{{\mathbf{a}}}
\def\rvb{{\mathbf{b}}}

\def\rvg{{\mathbf{g}}}

\def\rvm{{\mathbf{m}}}

\def\rvv{{\mathbf{v}}}

\def\erva{{\textnormal{a}}}
\def\ervb{{\textnormal{b}}}

\def\rmG{{\mathbf{G}}}

\def\vtheta{{\bm{\theta}}}

\DeclareMathAlphabet{\mathsfit}{\encodingdefault}{\sfdefault}{m}{sl}
\SetMathAlphabet{\mathsfit}{bold}{\encodingdefault}{\sfdefault}{bx}{n}

\def\gX{{\mathcal{X}}}

\def\sA{{\mathbb{A}}}

\newcommand*\circled[1]{\tikz[baseline=(char.base)]{
            \node[shape=circle,draw,inner sep=1pt] (char) {\tiny{#1}};}}

\newcommand{\R}{\mathbb{R}}

\DeclareMathOperator*{\argmin}{arg\,min}

\makeatletter
\def\thickhline{%
  \noalign{\ifnum0=`}\fi\hrule \@height \thickarrayrulewidth \futurelet
   \reserved@a\@xthickhline}
\def\@xthickhline{\ifx\reserved@a\thickhline
               \vskip\doublerulesep
               \vskip-\thickarrayrulewidth
             \fi
      \ifnum0=`{\fi}}
\makeatother

\newlength{\thickarrayrulewidth}
\usepackage{arxiv}
\usepackage{amsmath}
\usepackage{amssymb}
\usepackage[utf8]{inputenc} 
\usepackage[T1]{fontenc}    
\usepackage{hyperref}       
\usepackage{url}            
\usepackage{booktabs}       
\usepackage{amsfonts}       
\usepackage{nicefrac}       
\usepackage{microtype}      
\usepackage{lipsum}
\usepackage{xspace}
\usepackage{breqn}
\usepackage{algorithm}
\usepackage{algorithmic}
\usepackage{graphicx}
\usepackage{amsthm}
\usepackage{mathrsfs}
\usepackage{mathtools}
\usepackage{tikz}
\usepackage{multicol}
\usepackage{multirow}
\usepackage{paralist, tabularx} 

\newtheorem{lemma}{Lemma}[section]
\newtheorem{theorem}[lemma]{Theorem}
\newtheorem{corollary}[lemma]{Corollary}

\newtheorem*{remark*}{Remark}
\newtheorem{assumption}{Assumption}

\makeatletter
\newcommand\mathcircled[1]{%
  \mathpalette\@mathcircled{#1}%
}
\newcommand\@mathcircled[2]{%
  \tikz[baseline=(math.base)] \node[draw,circle,inner sep=1pt] (math) {$\m@th#1#2$};%
}

\newcommand{\method}{ACMo\xspace}

\title{\method: Angle-Calibrated Moment Methods for Stochastic Optimization}

\author{
  Xunpeng Huang\\
  Bytedance AI Lab\\
  \texttt{huangxunpeng@bytedance.com}
    \And
  Runxin Xu\\
  Peking University\\
  \texttt{runxinxu@gmail.com}
    \And
  Hao Zhou\\
  Bytedance AI Lab\\
  \texttt{zhouhao.nlp@bytedance.com}
    \And
  Zhe Wang\\
  Ohio State University\\
  \texttt{wang.10982@osu.edu}
    \And
  Zhengyang Liu\\
  Beijing Institute of Technology\\
  \texttt{zhengyang@bit.edu.cn}
    \And
  Lei Li\thanks{Corresponding author}\\
  Bytedance AI Lab\\
  Beijing, China\\
  \texttt{lileilab@bytedance.com}
}

\begin{document}
\maketitle

\begin{abstract}
Due to its simplicity and outstanding ability to generalize, stochastic gradient descent (SGD) is still the
most widely used optimization method despite its slow convergence. 
Meanwhile, adaptive methods have attracted rising attention of optimization and machine learning communities, both for the leverage of life-long information and for the profound and fundamental mathematical theory.
Taking the best of both worlds is the most exciting and challenging question in the field of optimization for machine learning.
Along this line, we revisited existing adaptive gradient methods from a novel perspective, refreshing understanding of second moments. 
Our new perspective empowers us to attach the properties of second moments to the first moment iteration, and to propose a novel first moment optimizer, \emph{Angle-Calibrated Moment method} (\method). 
Our theoretical results show that \method is able to achieve the same convergence rate as mainstream adaptive methods. 
Furthermore, extensive experiments on CV and NLP tasks demonstrate that \method has a comparable convergence to SOTA Adam-type optimizers, and gains a better generalization performance in most cases.
\end{abstract}

\section{Introduction}
\label{sec:intro}
Deep neural network has been widely adopted in different applications because of its excellent performance, which always requires a huge amount of data for training.
Calculating the full gradient of data and performing the full gradient descent (GD) become computationally expensive. 
Therefore, stochastic gradient descent (SGD) has become very popular for training deep neural networks. 
Empirically, in each step of the training, SGD samples a mini-batch of data and applies gradient descent with the corresponding stochastic gradients computed on the mini-batch.

In practice however, the vanilla SGD does not always produce good results, in which case many SGD variants are proposed. 
Especially, relevant work has shown that incorporating momentum information into SGD can help with its optimization process. 
Specifically, by introducing the first moment, the SGD momentum can help the model escape from some saddle points, and therefore improve its generalization. 
Intuitively, vanilla SGD walks along the steepest path, whereas the added momentum makes the optimization process smoother and quicker, thus helping the model to barrel through narrow valleys.

Additionally, \textit{second moments} are used to adapt the learning rate of model parameters~\cite{Duchi2011}, performing smaller updates (i.e. low learning rates) for parameters associated with frequently occurring features and larger updates (i.e. high learning rates) for those associated with infrequent features. 
This always accelerates the optimization process towards the objective.
Moreover, Adam is proposed to utilize both of first and second moments for enjoying both of their benefits~\cite{Kingma2014}.
Currently, it is one of the most widely used methods for neural network training. 

Although Adam has achieved a great success, we argue that introducing extra second moments is not necessarily the best way to boost the optimization efficiency.
First, by designing special optimization problems, the solution of adaptive gradient methods with second moments may fall into a local minima of pool generalization~\cite{wilson2017marginal}.
Second, keeping a copy of second moments~(with the same size as the parameters) brings high memory overhead, leading to smaller mini-batch size for training.
This may adversely affect the performance of many applications that are sensitive to training batch size.
Given the above concerns, could we design an optimization approach that \emph{only uses first moments but enjoys the benefits of both first~(better generalization by escaping saddle points) and second~(fast convergence with adaptive learning) moments}?

In this paper, we offer an affirmative to the question by proposing an Angle-Calibrated Moment method~(\method) for stochastic optimization. \method takes a further step by explicitly requiring that the opposite direction of current descent be in acute angles with both the current gradient and the directions of historical mini-batch updating.
Given the angle constraint, \method is likely to ensure descents for all mini-batch losses, while guaranteeing sufficient descent of current mini batches.
Although \method has abandoned the second moment, on which many other methods rely, it still takes the advantage of fast convergence as those methods.
We summarize our contributions below
\begin{compactitem}
\item We propose \method, which is a new SGD variant. To the best of our knowledge, \method is the first approach 
without relying on second moments, but still has comparable convergence speed as compared with Adam-type methods.
\item We provide a novel view and an intuitive analysis to understand the effect of second moments, which may shed light on the following work in this field . 
\item We provide theoretical results for the gradient norm convergence of \method on the nonconvex settings, which illustrates that \method can offer the same convergence rate comparing with the Adam-type optimizers. Experimental results on different CV and NLP tasks show that, even without second moments, \method can display convergence speed that is on par with SOTA Adam-type optimizers, while obtains even better generalization in most cases. 
\end{compactitem}

\textbf{Notations.}
In the rest of this paper, we have parameters $\vtheta \in \R^d$ where $d$ denotes the dimension of parameters.
The $l_2$ norm of a given vector $\vtheta$ is expressed by $\|\vtheta\| = \sqrt{\sum_{i=1}^d \theta_i^2}$.
With slightly abuse of notation, we represent arithmetic symbols as element-wise operations for vectors, e.g., $\rva^2 = [\erva_1^2, \erva_2^2,\ldots]^T, \rva/\rvb=[\erva_1/\ervb_1, \erva_2/\ervb_2,\ldots]^T.$
We denote  $\lfloor x\rfloor$ as the greatest integer less than or equal to the real number $x$. Given any integers $x, y$, where $y>0$, we denote $x \pmod y$ as the remainder of the Euclidean division of $x$ by $y$.
In the finite-sum loss function, $f(\vtheta) = \frac{1}{n}\sum_{i=1}^n f_i(\vtheta)$, the number of instances and the loss of the $i$-th training data are represented as $n$ and $f_i(\vtheta)$, respectively.
Besides, we denote $f_{\sA}(\vtheta)$ when we feed a collection of samples, i.e., $f_\sA(\vtheta) := \frac{1}{\left|\sA\right|}\sum_{i\in \sA}f_i(\vtheta)$. For an optimization algorithm, if its update paradigm can be formulated as
\begin{equation}
    \label{eq:iteration_auxiliary_problem}
    \vtheta_{t+1} = \argmin_{\vtheta}\ \hat{f}_{\sA_t}(\vtheta, \vtheta_t),
\end{equation}
we denote eq.~\ref{eq:iteration_auxiliary_problem} as the iteration auxiliary problem (similar to \cite{Nesterov2006Cubic}).

\section{Related Work}
\label{sec:related}
In this section, we will introduce the development of the neural network optimizers.

\textbf{First Moment Optimizers} SGD-momentum and Nesterov accelerated gradient~\cite{nesterov2013introductory} are widely used in training large-scale neural networks, but because of the learning rate issue, their excellent generalization ability is brittle.

\textbf{Second Moment Optimizers} To accelerate the convergence, researchers began to focus on the design of adaptive gradient methods for a fast and simple optimizer. 
Adagrad~\cite{Duchi2011} introduced the second moment to obtain a self-adaptive learning rate, thus freeing researcher of the troubles of parameter tuning. 
The update rules of Adagrad can be formulated as $\vtheta_{t+1} = \vtheta_{t}-\alpha_{t}\cdot \rvg_t/\sqrt{\rvv_t}$,
where $\rvg_t$ denotes the stochastic gradient, $\rvv_t$ is the accumulation of gradient's second moments, i.e., $\rvv_t =\sum_{\tau=1}^t\rvg_\tau^2$, and $\alpha_t $ is the decreasing learning rate with $\alpha_t = {\Theta} (1/\sqrt{t})$.
Theoretically, Adagrad improved the convergence of regret from $O(\sqrt{d/T})$ to $O(1/\sqrt{T})$ for the convex objectives with sparse gradients. 
However, in practice, people realize that adaptive gradient of Adagrad, i.e., ${\rvg_t}/{\sqrt{\rvv_t}}$ goes to zero very quickly due to the fact that $\rvv_t$ accumulates to large number quickly as the algorithm proceeds, and they often require optimizers to have a lower memory cost for training a larger mini-batch. 
To make it through, RMSProp~\cite{Hinton2012} uses the exponential decay in the second moment to control the accumulation speed of second moment in Adagrad, and min-max squared graidient is introduced to implement Adagrad with a memory-efficient way~\cite{anil2019memory}.

\textbf{Adam-type Optimizers}
To take both the benefits from first and second moments, Adam was proposed to incorporate the momentum into RMSProp. 
The detailed procedure of Adam can be formulated as $\vtheta_{t+1} = \vtheta_{t}-\alpha_{t}\cdot \rvm_t/\sqrt{\rvv_t}$, where $\rvm_t$ is the exponential decay of momentum, i.e., $\rvm_t = \sum_{\tau=1}^t(1-\beta_1)\beta_1^{t-\tau}\rvg_\tau$, and $\rvv_t$ is the exponential decay of second moments, i.e., $\rvv_t =
\sum_{\tau=1}^t(1-\beta_2)\beta_2^{t-\tau}\rvg_\tau^2$.

The faster convergence, robust hyper-parameters and good performance on bunch of tasks make Adam become one of the most successful optimizers these years.
Adam was proved not to be convergent in certain convex cases, and Amsgrad was proposed to correct the direction of Adam~\cite{Reddi2019}.

Although convergence, the generalization ability of adaptive algorithms is worse than SGD-momentum in many tasks.
Thus, a lot of studies are proposed to improve the generalization performance of Adam-type methods by making some connections between Adam and SGD-momentum, e.g., ND-Adam~\cite{zhang2017normalized}, AdamW~\cite{Loshchilov2017}, SWATS~\cite{keskar2017improving}, Adabound~\cite{Luo2019}, PAdam~\cite{Chen2018}, etc.

All the aforementioned methods try to find  the connection between Adam and SGD-momentum, thus all these algorithms take the second moment adaptation as a grant. 
Different from previous work, we revisit the original idea of second moment adaptation, and propose an angle based algorithm without second moments.
Such intuition makes \method have a simpler update, a lower memory overhead, a good generalization performance, and a comparable convergence to SOTA Adam-type optimizers.

\section{Our Proposed \method}
\label{sec:method}
In this section, we propose a novel optimizer, the Angle-Calibrated Moment method (\method), for both fast convergence and ideal generalization \emph{with only first moments}. 

We first give some theoretical analysis, showing that incorporating second moments into 
optimization approximately equals to penalize the projection of the current descent direction on previous gradients.
Besides, such penalty helps facilitate the descent of the current batch loss while not harming the descent effects of previous iterations as much as possible~(section \ref{sec:3_1}).
After that, in section~\ref{sec:algorithm_detail}, we replace the projection penalty with the inner product penalty in iteration auxiliary problems, which partially preserve geometric properties of second moments (protecting effects of previous updates) to expect a fast convergence.
Besides, the new iteration auxiliary problem can be considered as a general form of SGD-momentum updates to expect a good generalization ability.
Finally, with sufficient descent constraints of the current batch loss, we propose \method whose update paradigm is 
\begin{equation}
    \vtheta_{t+1} = \vtheta_t - \alpha_t\left(\rvg_t + \beta_t \cdot \frac{\left\|\rvg_t\right\|}{\left\|\hat{\rvm}_{t-1}\right\|+\delta_t}\cdot \hat{\rvm}_{t-1}\right),
\end{equation} 
where $\rvg_t$ and $\rvm_t$ are mini-batch gradients and angle-calibrated moments at iteration $t$, respectively.

\subsection{Second Moments Work as Projection Penalty to Preserve Previous Descent}
\label{sec:3_1}
In this section, we revisit the iterations of existing adaptive algorithms from a novel point of view and denote the essential effect of second moments which is to penalize the projection of the current descent direction on previous gradients.
Such penalty is designed to decrease the cumulative loss by guaranteeing a descent on the current batch loss without increasing previous batch loss.

One may notice that almost all the adaptive methods utilize second moments to adjust the individual magnitude of their updates.
However, the efficient calculation of second moments is only proposed in Adagrad~\cite{Duchi2011}. 
Here is the original update of Adagrad:
\begin{equation}
\label{eq:Adagrad_original_update}
\vtheta_{t+1} = \vtheta_{t}-\alpha_{t}\rmG_t^{-1/2}\rvg_t,
\end{equation} 
where $\rmG_t = \sum_{\tau = 1}^{t} \rvg_\tau\rvg_\tau^T$, and $\rvg_\tau$ denotes the stochastic gradient calculated at iteration $\tau$.
The iteration auxiliary problem of Adagrad corresponding to eq.~\ref{eq:Adagrad_original_update} can be formulated  as
\begin{equation}
\label{eq:Adagrad_proximal_subproblem}
\begin{split}
    \vtheta_{t+1} = \mathop{\argmin}\limits_{\vtheta} \overline{f}(\vtheta),\quad \overline{f}(\vtheta) \coloneqq \underbrace{\left(\vtheta-\vtheta_t\right)^T\rvg_t}_{T_1} + \underbrace{\frac{1}{2\alpha_t}\left(\vtheta-\vtheta_t\right)^T\rmG_t^{1/2}\left(\vtheta-\vtheta_t\right)}_{T_2}.
\end{split}
\end{equation}
To investigate the properties of second moments, we provide $\overline{f}(\vtheta)$ an upper bound, $\hat{f}_{\sA_t}(\vtheta)$, with the Young's inequality~\cite{young1912}, and minimize it with another iteration auxiliary problem formulated as
\begin{equation}
    \label{eq:Adagrad_objectiv_upperbound}
    \begin{split}
    \hat{\vtheta}_{t+1} = \mathop{\argmin} \hat{f}_{\sA_t}(\vtheta), \ \hat{f}_{\sA_t}(\vtheta)= \underbrace{\left(\vtheta-\vtheta_t\right)^T\rvg_t}_{T_1} +\underbrace{\frac{1}{2\alpha_t}\left\|\vtheta-\vtheta_t\right\|^2}_{T_3} + \underbrace{\frac{1}{8\alpha_t}\sum_{\tau=1}^{t}\left\|\left(\vtheta-\vtheta_t\right)^T\rvg_\tau\right\|^2}_{T_4}
    \end{split}
\end{equation}
where $\sA_t$ denotes the chosen mini-batch at iteration $t$.
Thus, the global minima $\hat{\vtheta}_{t+1}$ of eq.~\ref{eq:Adagrad_objectiv_upperbound} can also be considered as an approximate solution of eq.~\ref{eq:Adagrad_proximal_subproblem}.
Investigating $\hat{f}_{\sA_t}(\vtheta)$, we find $T_1+T_3$ is the second-order Taylor polynomial of $f_{\sA_t}(\vtheta)$ near $\vtheta_t$, which is completely the same with the objective function of SGD iteration auxiliary problem.
Besides, the term $T_4$ in $\hat{f}_{\sA_t}(\vtheta)$ can be considered as a penalty for the projections of current descent $\vtheta - \vtheta_t$ on the previous gradients $\rvg_\tau$. 
If we approximate $\rvg_\tau$ in $T_4$ with $\nabla f_{\sA_\tau}(\vtheta_t)$ due to smoothness assumption, and replace the projection regularization ($T_4$) to some constraints, we can obtain an approximate optimization problem with a hard margin formulated as
\begin{equation}
    \label{eq:acute_subproblem}
    \begin{split}
    \argmin\limits_{\vtheta} \  &\left(\vtheta-\vtheta_t\right)^T\rvg_t + \frac{1}{2\alpha_t}\left\|\vtheta-\vtheta_{t}\right\|^2 \\
    \text{s.t.}\ &\left(\vtheta-\vtheta_t\right)^T\nabla f_{\sA_\tau}(\vtheta_t)=0, \tau \le t.
    \end{split}
\end{equation}
The constraints in eq.~\ref{eq:acute_subproblem} denote the descent direction, i.e., $\vtheta_{t+1}-\vtheta_{t}$ is desire to be orthogonal to $\nabla f_{\sA_\tau}(\vtheta_t)$ for any $\tau\le t$, which plays a similar role to $T_4$ in eq.~\ref{eq:Adagrad_objectiv_upperbound}. In the following, we provide a explanation about such constraint optimization problem from a random shuffling perspective.

Before each epoch begins, the whole training dataset $\sA$ is usually randomly shuffled, and is partitioned into mini-batch of equal size $\left\{\sA_0, \sA_1,\ldots, \sA_{p-1}\right\}$.
Then the algorithm is fed with the samples in the fixed order, say, first $\sA_0$, then $\sA_1$, and so on.  
The whole procedure repeats after each iteration over the whole dataset.
Let $\nabla f_{\sA_t}(\vtheta_t)$ be the gradient calculated at iteration $t$ by using the sample in subset $\sA_t$, where $0 \leq t < p$.  
Note that the loss function is the average of all the samples, e.g., $\frac{1}{n}\sum_{i=1}^{n} f_i(\vtheta)$.
If we utilize $\nabla f_{\sA_t}(\vtheta_t)$ to directly update parameters like SGD, e.g., $\vtheta_{t+1} = \vtheta_{t} - \alpha \nabla f_{\sA_{t}}(\vtheta_t)$, the batch loss $\sum_{{i \in \sA_t}}  f_i(\vtheta)$ will decrease since it aligns with the opposite direction of its gradient, e.g., $ (\vtheta_{t+1}-\vtheta_{t})^T \nabla f_{\sA_{t}}(\vtheta_t) = - \|\nabla f_{\sA_{t}}(\vtheta_t)\|^2 <0$. 
However, for the loss corresponding to the sample not in $\sA_t$, e.g., $f_{\sA_\tau}(\vtheta),\ \tau\not=t$,  it is highly possible that  $(\vtheta_{t+1}-\vtheta_{t})^T \nabla f_{\sA_{\tau}}(\vtheta_t) >0$. 
In other words, only using $-\nabla f_{\sA_{t}}(\vtheta_t)$ as update direction will decrease the loss corresponding to $\sA_{t}$ but increase the loss except $\sA_{t}$.
Hence, with the orthogonal requirements, i.e., the constraints in eq.~\ref{eq:acute_subproblem}, one can consider that {\em Adagrad decreases the cumulative loss by guaranteeing a descent on the current batch loss while does not increase previous batch loss sharply.}

\subsection{Angle-Calibrated Moments Warrant Descents}
\label{sec:algorithm_detail}
In this section, we first enhance the iteration auxiliary problem eq.~\ref{eq:Adagrad_objectiv_upperbound}. Then, we make the opposite direction of current descent forms acute angles with both current mini-batch gradient and some moments to introduce our \method.

\begin{figure}[H]
    \centering
    \includegraphics[width=0.4\textwidth]{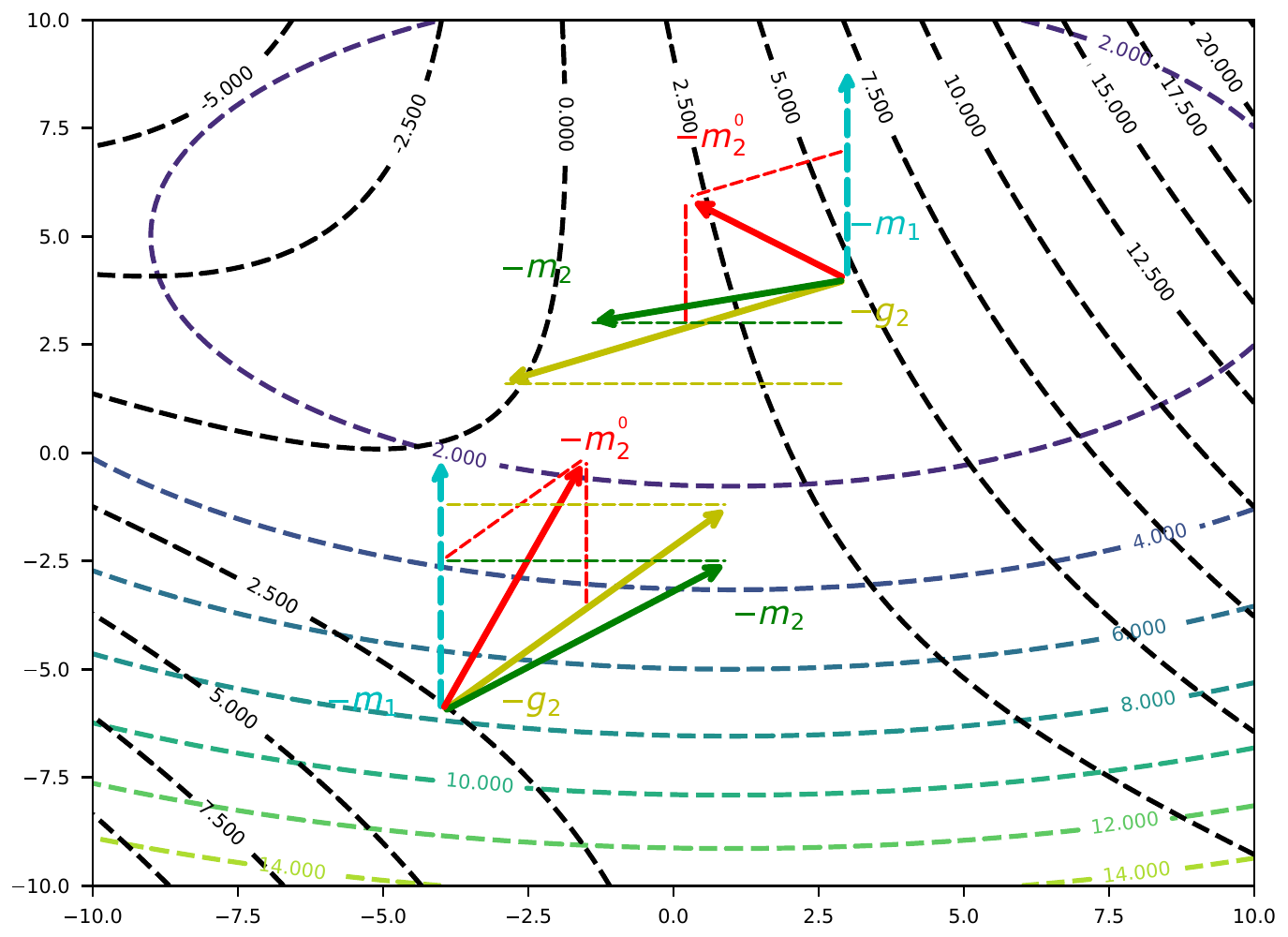}
    \vspace{-15pt}
    \caption{\scriptsize The cyan, yellow, green and red vectors respectively denote previous moments, current mini-batch (the objective corresponding to black dashes) gradients,  descent directions for iteration auxiliary problems with the projection or ACMo regularization.}
    \label{fig:acmo_intuition}
\end{figure}

\begin{algorithm}[H]
    \caption{\textbf{A}ngle-\textbf{C}alibrated \textbf{Mo}ment method.}
    \label{alg:acutum}
 \begin{algorithmic}[1]
    \STATE {\bfseries Input:} initial point $\vtheta_0\in \gX$; step size $\left\{\alpha_t\right\}$, momentum parameters $\left\{\beta_t\right\}$
    \STATE set $\hat{\rvm}_0$ = 0
    \FOR{$t = 1$ to $T$}
         \STATE $\rvg_t = \nabla f_{\sA_t}(\vtheta_t)$
         \STATE $\hat{\beta}_t = \beta_{t}\left\|\rvg_t\right\|\big/ \left(\left\|\hat{\rvm}_{t-1}\right\|+\delta_t\right)$
         \STATE $\hat{\rvm}_{t} = \rvg_t+ \Psi(\hat{\beta}_t, \hat{\beta}_{t-1}) \cdot \hat{\rvm}_{t-1}$
         \STATE $\vtheta_{t+1} = \Pi_{\gX}\left(\vtheta_t  - \alpha_t\cdot \hat{\rvm}_{t}\right)$
         \ENDFOR
         \STATE {\bfseries Return:} $\vtheta_{\mathrm{o}}$ with a discrete distribution as \\ $P(\mathrm{o}=i) = \alpha_{i-1}\Big/\left(\sum_{\tau=1}^{T-1}\right)\alpha_{\tau},\ 2\le \mathrm{o}\le T$. 
 \end{algorithmic}
 \end{algorithm}

Enlightened by the analysis in section~\ref{sec:3_1}, we realize that adding the projection penalty ($T_4$ in eq.~\ref{eq:Adagrad_objectiv_upperbound}) is to ensure that the descent direction does not increase previous mini-batch losses.
If we replace it with a weighted inner products penalty, we can even expect the descent direction make decrease of  both the current mini-batch loss and the previous mini-batch losses.
Hence, we can formulate the objective of the new iteration auxiliary problem as:
\begin{equation}
    \label{eq:acutum_proximal_subproblem}
    \begin{split}
        \tilde{f}_{\sA_t}(\vtheta)=\left(\vtheta-\vtheta_t\right)^T\rvg_t + \frac{L_t}{2}\left\|\vtheta- \vtheta_t\right\|^2 + \left(\vtheta-\vtheta_t\right)^T\hat{\rvg}_t,\quad   \hat{\rvg}_t \coloneqq \hat{\beta}_t\sum_{\tau=1}^{t}w_\tau\rvg_\tau,
    \end{split}
\end{equation}
where $L_t$ denotes the constant of smoothness, and $w_\tau$ denotes the weight to measure the approximation confidence about $\rvg_\tau$. 
For updating $\vtheta$ through minimizing eq.~\ref{eq:acutum_proximal_subproblem}, the optimum of the quadratic function satisfies
\begin{equation}
    \label{eq:update_paradigm_general}
    \frac{\partial\tilde{f}_{\sA_t}(\vtheta) }{\partial\vtheta}|_{\vtheta=\vtheta_{t+1}}= \rvg_t + L_t(\vtheta_{t+1} - \vtheta_t)+\hat{\rvg}_t= 0 \Leftrightarrow \vtheta_{t+1} =\vtheta_t - \frac{\hat{\rvg}_t + \rvg_t}{L_t}.
\end{equation}
Notice that if we set $\hat{\beta}_t =1$ and $w_\tau = \beta_{0}^\tau$ in eq.~\ref{eq:acutum_proximal_subproblem}, eq.~\ref{eq:update_paradigm_general} corresponds to the update paradigm of SGD-momentum coincidentally.
Thus, we may get the benefit of generalization performance from the iteration auxiliary problems.
Different from SGD-momentum, we want to guarantee a sufficient descent for the loss of current mini-batch. Hence, $\hat{\beta}_t$ is requested to have 
\begin{equation}
    \label{ineq:sufficient_descent_guarantee}
    \begin{split}
        \hat{\beta}_t:=\frac{\beta_t\left\|\rvg_t\right\|}{\left\|\sum_{\tau=1}^{t-1}w_\tau \rvg_\tau\right\| + \delta_t},\ \beta_t\le 1 \Rightarrow
        \hat{\beta}_t\le \frac{\left\|\rvg_t\right\|}{\left\|\sum_{\tau=1}^{t-1}w_\tau \rvg_\tau\right\|}
        \Rightarrow
        \hat{\beta}_t\left\|\sum_{\tau=1}^{t-1}w_\tau \rvg_\tau\right\| \le \left\|\rvg_t\right\|,
    \end{split}
\end{equation}
for Eq.~\ref{eq:acutum_proximal_subproblem}.
Notice that the last inequality of eq.~\ref{ineq:sufficient_descent_guarantee} implies current mini-batch loss descent as
\begin{equation*}
    \begin{aligned}[b]
        f_{\sA_t}(\vtheta_{t+1}) - f_{\sA_t}(\vtheta_t) \mathop{\le}^{\circled{1}} & (\vtheta_{t+1} - \vtheta_t)^T\rvg_t + \frac{L_t}{2}\|\vtheta_{t+1} - \vtheta_t\|^2\\ 
        \mathop{=}^{\circled{2}} & -\left[\frac{\|\rvg_t\|^2+\hat{\rvg}_t^T\rvg_t}{L_t}\right] + \frac{L_t}{2}\cdot\frac{\|\rvg_t+\hat{\rvg}_t\|^2}{L_t^2} = -\frac{\|\rvg_t\|^2}{2L_t} + \frac{\|\hat{\rvg}_t\|^2}{2L_t} \mathop{\le}^{\circled{3}} 0,
    \end{aligned}
\end{equation*}
where $\circled{1}$ follows from the smoothness assumption, $\circled{2}$ is established due to eq.~\ref{eq:update_paradigm_general} and  $\circled{3}$ is from eq.~\ref{ineq:sufficient_descent_guarantee}.
Besides, we denote $\sum_{\tau=1}^{t-1}w_\tau \rvg_\tau$ as $\hat{\rvm}_{t-1}$ with the following iteration to generate $w_\tau$s automatically,
\begin{equation}
    \hat{\rvm}_{t} = \rvg_t + \beta_t\cdot\frac{\left\|\rvg_t\right\|}{\left\|\hat{\rvm}_{t-1}\right\|+\delta_t}\cdot \hat{\rvm}_{t-1},
\end{equation}
where $\beta_t$ and $\delta_t$ are hyper-parameters. 
From a geometric perspective, $\hat{\rvm}_t$ is the angle bisector of $\rvg_t$ and $\hat{\rvm}_{t-1}$ coincidentally when $\beta_t = 1$ and $\delta_t=0$ (see Figure~\ref{fig:acmo_intuition}).
Then, we can reformulate the iteration auxiliary problem of \method as:
\begin{equation*}
    \begin{split}
    \vtheta_{t+1} = \mathop{\argmin}\limits_{\vtheta}\ &\frac{L_t}{2}\left\|\vtheta-\vtheta_{t}\right\|^2+\left(\vtheta-\vtheta_t\right)^T\rvg_t+ \beta_t\cdot\frac{\left\|\rvg_t\right\|}{\left\|\hat{\rvm}_{t-1}\right\|+\delta_t}\cdot\left(\vtheta-\vtheta_t\right)^T\hat{\rvm}_{t-1},
    \end{split}
\end{equation*}
where we can obtain the update paradigm of \method as:
\begin{equation}
    \vtheta_{t+1} = \vtheta_t - \alpha_t\left(\rvg_t + \beta_t \cdot \frac{\left\|\rvg_t\right\|}{\left\|\hat{\rvm}_{t-1}\right\|+\delta_t}\cdot \hat{\rvm}_{t-1}\right).
\end{equation} 

In summary, we observe that 
(i) if $\vtheta_t - \vtheta_{t+1}$ forms acute angles with both $\rvg_t$ and $\nabla f_{\sA_\tau}(\vtheta_t)$, rather than penalizing the projection as Adagrad, we can obtain descent on both current and previous batches;
(ii) to handle the case when the estimation of accumulative weighted gradients using $\hat{\rvm}_{t-1}$ is not accurate, we expect the current gradient $\rvg_t$ to dominate the descent direction for guaranteeing a sufficient descent of current mini-batch loss.
Hence, we propose a new first moment optimizer, which attach the properties of the second moment to first moment iterations, inspired by the iteration auxiliary problem of Adagrad. 
The proposed algorithm \method is shown in Algorithm~\ref{alg:acutum}.
Note that, $\Psi(\cdot)$ is a function to guarantee a sufficient descent of the iteration auxiliary problem. In practice, we usually set $\Psi(\hat{\beta}_t, \hat{\beta}_{t-1}) = \hat{\beta}_t$. 

\section{Theoretical Results}
\label{sec:analysis}
In this section, we provide the convergence about gradient norm in expectation for our \method in nonconvex settings. 
Our theoretical results show that \method obtains the same convergence rate with Adam-type optimizers. All details of our proof can be found in our supplementary material. 

We list assumptions required in convergence analysis, and then provide the main theoretical results.
\begin{assumption}
  \label{ass:gradient_lipschitz_continuity}
  We assume the loss function $f(\vtheta)$ is differentiable, and has $L$-Lipschitz gradient, i.e., for any feasible solution $\vtheta_i,\vtheta_{j}\in \mathbb{R}^{d}$, $\|\nabla f_i(\vtheta_i) - \nabla f_j(\vtheta_j)\| \le L\|\vtheta_i-\vtheta_{j}\|$.
\end{assumption}
\begin{assumption}
  \label{ass:bounded_minimization}
  We assume the objective $f(\vtheta)$ is lower bounded, which means $ \min_{\vtheta}\ f(\vtheta)> -\infty$.
\end{assumption}
\begin{assumption}
  \label{ass:bounded_gradient_variance}
  For the mini-batch loss $f_{\sA_t}(\vtheta)$ at iteration $t$, we assume stochstic gradients, e.g., $\nabla f_{\sA_t}(\vtheta)$ and $\nabla f_i(\vtheta)$ satisfy
  \begin{equation*}
    \begin{split}
        &\mathbb{E}\left[\nabla f_{\sA_t}(\vtheta)\right] = \nabla f(\vtheta),\quad \left\|\nabla f_{\sA_t}(\vtheta)\right\|\le G,\quad \max_i\left\{\left\|\nabla f_i(\vtheta)-\nabla f(\vtheta)\right\|^2\right\}\le \sigma^2.
    \end{split}
  \end{equation*}
\end{assumption}
\begin{theorem}
\label{thm:gradient_norm_expectation_bound}
	Suppose Assumption~\ref{ass:gradient_lipschitz_continuity},~\ref{ass:bounded_minimization},~\ref{ass:bounded_gradient_variance} hold. If we set $\epsilon \ge 0$, $\delta\ge \sigma$, $\beta_t \le \frac{1}{50}$ and
	\begin{equation*}
		\begin{split}
			\alpha_t \le &\frac{3}{\left(4L+1240\right)\sqrt{t}},\quad \Psi(\hat{\beta}_{t+1}, \hat{\beta}_{t}) = \min\left\{\hat{\beta}_{t+1}, \sqrt{\frac{t+1}{t}}\hat{\beta}_{t}\right\}
		\end{split}
	\end{equation*}
	without loss of generality.  Then, the output of \method satisfies
	\begin{equation*}
		\mathbb{E}\left[\left\|\nabla f(\vtheta_o)\right\|^2\right]\le \frac{C_0}{\sqrt{T}} +   \frac{C_1\log(T)}{\sqrt{T}}
	\end{equation*}
	where $C_0$ and $C_1$ are constants independent with $T$, $n$, $d$ and presented in our proof.
\end{theorem}
Theorem~\ref{thm:gradient_norm_expectation_bound} shows that \method has the same convergence rate as Adam-type optimizers~\cite{chen2018convergence}.
Comparing with theoretical results provided in~\cite{Chen2018}, ours has additional $O(\log T/\sqrt{T})$ term. 
Since we do not require the condition about the sparsity of gradients as~\cite{chen2018convergence}. 
Besides, we show that rather than requesting the coefficient sequence of first moments ($\beta_{1,t}$ in Adam-type optimizers) to be non-increasing, a gently increasing sequence ($\Psi(\cdot)$ in \method) can keep an $\tilde{O}(1/\sqrt{T})$ convergence rate for ACMo.
This result expands the range of hyper-parameters selection.

\section{Experiments}
\label{sec:experiments}
In this section, we conduct extensive experiments on image classification and neural machine translation tasks.
We want to demonstrate the generalization performance and the efficiency of \method, as compared with other adaptive gradient methods, e.g, Adam~\cite{Kingma2014}, Amsgrad~\cite{Reddi2019}, Adamw~\cite{Loshchilov2017},  PAdam~\cite{Chen2018}, and Adabound~\cite{Luo2019} and SGD-momentum.

\textbf{Hyperparameter Tuning}
Hyperparameters in optimizers can exert great impact on ultimate solutions found by optimization algorithms. In our experiments, we tuned over hyperparameters in the following way.
For all optimizers in our experiments, we chose the best initial set of step size from $\left\{1e-1, 5e-2, 1e-2, 5e-3, \ldots, 5e-5\right\}$.
For \emph{SGD-momentum}, we tuned the coefficient of momentum from $\{0.9, 0.8, \ldots, 0.1\}$.
For \emph{Adam, AMSGrad, AdamW, PAdam, and Adabound}, we turn over $\beta_1$ values of $\left\{0.9, 0.99\right\}$ and $\beta_2$ values of $\{0.99, 0.999\}$ and the perturbation value $\epsilon = 1e-8$ in image classification tasks. Besides, we set $\beta_1$ and $\beta_2$ as the suggested values of Transformer~\cite{vaswani2017attention}, where $\beta_1 = 0.9$ and $\beta_2 = 0.98$ in nerual machine translation tasks. For PAdam, we chose the best hyper-parameter $p$ from $\left\{1/4, 1/8, 1/16\right\}$.
For \emph{\method}, we directly applied the default hyperparameters, i.e., $\beta_{t}=0.9$, $\Psi(\hat{\beta}_t, \hat{\beta}_{t-1}) = \hat{\beta}_t$, for all our experiments.

\begin{figure*}[!tb]
    \centering
    \includegraphics[width=1.0\textwidth]{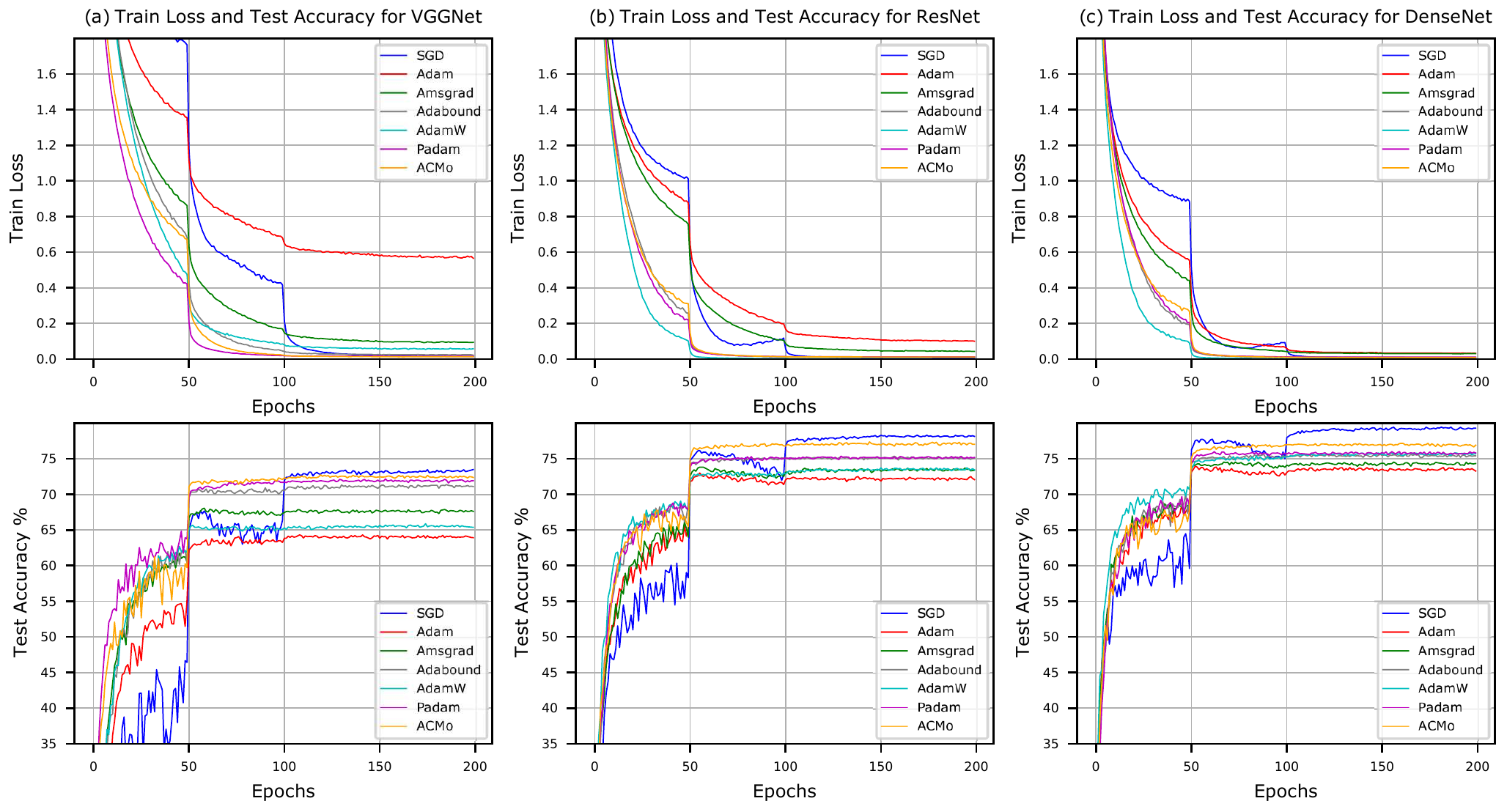}
    \vspace*{-20pt}
    \caption{\small Learning curves of optimizers for CNNs on CIFAR-100 image classification task. Top: training loss. Bottom:test accuracy.}
    \vspace*{-20pt}
    \label{fig:exp_fixed_weight_decay}
\end{figure*}

\subsection{Experiments on Image Classification Tasks}
In the image classification tasks, adaptive gradient methods display fast convergence, but poorer generalization results when compared to the well tuned SGD momentum, which includes proper hyper parameters for the learning and weight decay.
Note that introducing weight decay is equivalent to adding $l_2$ regularization to the objectives, and has a significant impact on the generalization ability of optimizers. 
Hence, our experiments were conducted from two perspectives as guidelines:
(i) We record the convergence of training loss and the test accuracy  of all optimizers with fixed weight decays (they optimize the same objective function). %
(ii) We adopt the optimal weight decays for all optimizers, then investigate their generalization ability based on the test accuracy. 

The paper then proceeds to introduce the experimental settings for the image classification tasks. 
We used two datasets CIFAR-10, CIFAR-100~\cite{krizhevsky2009learning}, and tested three different CNN architectures including VGGNet~\cite{Simonyan2014}, ResNet~\cite{He2016} and DenseNet~\cite{Huang2017}.
To achieve stable convergence, we ran 200 epochs, and set the learning rate to decay by $0.1$ every $50$ epochs.
We performed cross-validation to choose the best learning rates for all the optimizers and second moment parameters $\beta_2$ for all the adaptive gradient methods.

\textbf{Experiment with Fixed Weight Decay:}
We first evaluated CIFAR-100 dataset.
In this experiment, the values of weight decay in all optimizers were fixed, and were chosen to be the weight decay in SGD-momentum when it achieves the maximum test accuracy.

\emph{Though  \method was second only to SGD momentum in terms of generalization performance, it was significantly faster than not only the latter, but also Adam and AMSGrad in terms of convergence speed. Its convergence speed is comparable with SOTA Adam variants.}
From the first row of Figure~\ref{fig:exp_fixed_weight_decay}, i.e, the training curves of three tests, we observed that \method significantly outperforms Adam, and has comparable rate with PAdam and Adabound when we fixed the weight decay.
Also, in the second row of  Figure~\ref{fig:exp_fixed_weight_decay}, i,e, the test accuracy, our \method outperformed all benchmarks except only SGD-momentum.
However, the  training loss converges much slower when implemented with SGD-momentum which is also recorded in other literature.

\begin{figure*}[!tb]
    \centering
    \includegraphics[width=1.0\textwidth]{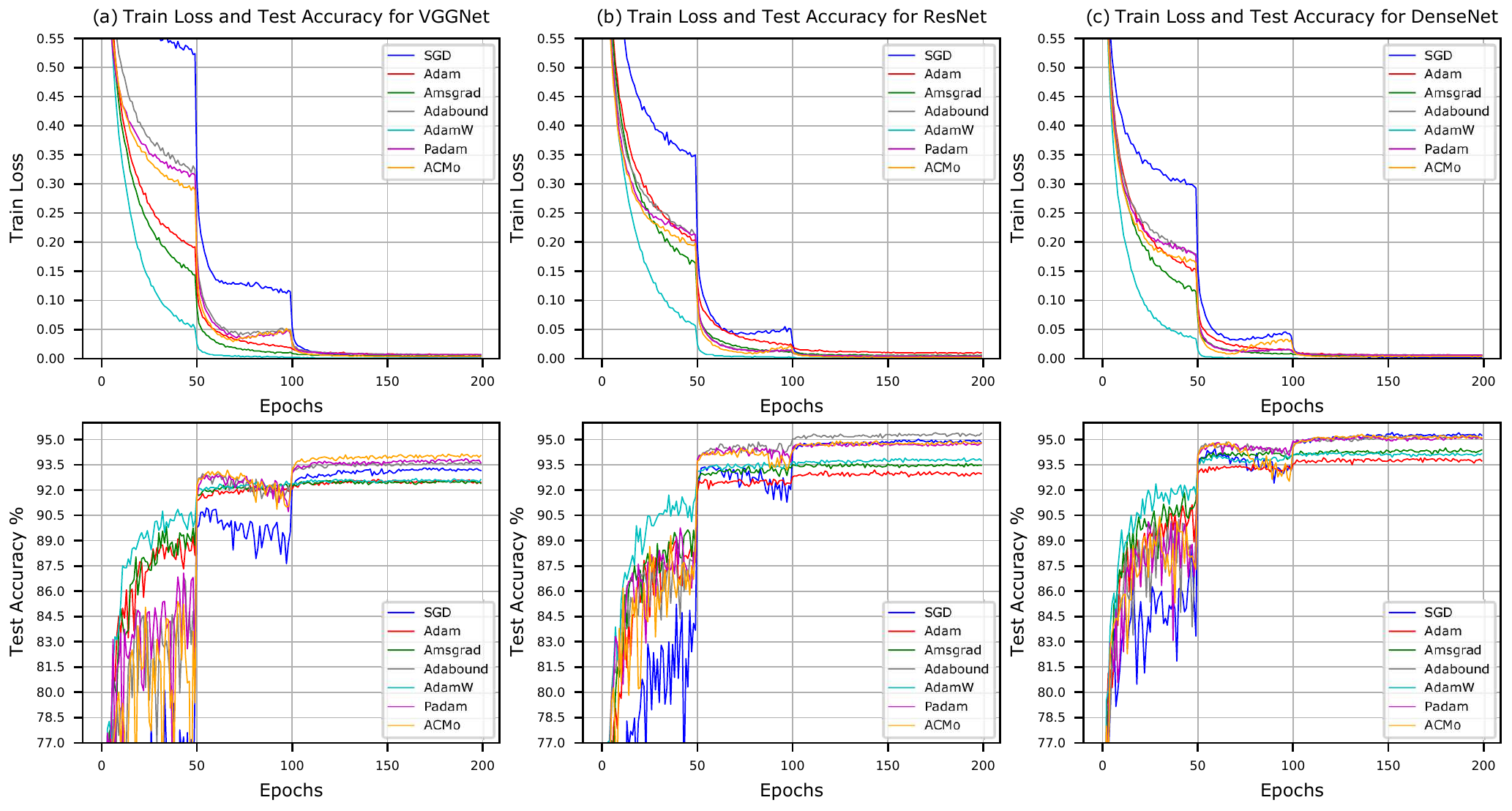}
    \vspace*{-20pt}
    \caption{\small Learning curves of optimizers for CNNs on CIFAR-10 image classification task. Top: training loss. Bottom:test accuracy.}
    \vspace*{-10pt}
    \label{fig:optim_weight_decay}
\end{figure*}

\begin{figure*}[!tb]
    \centering
    \includegraphics[width=1.0\textwidth]{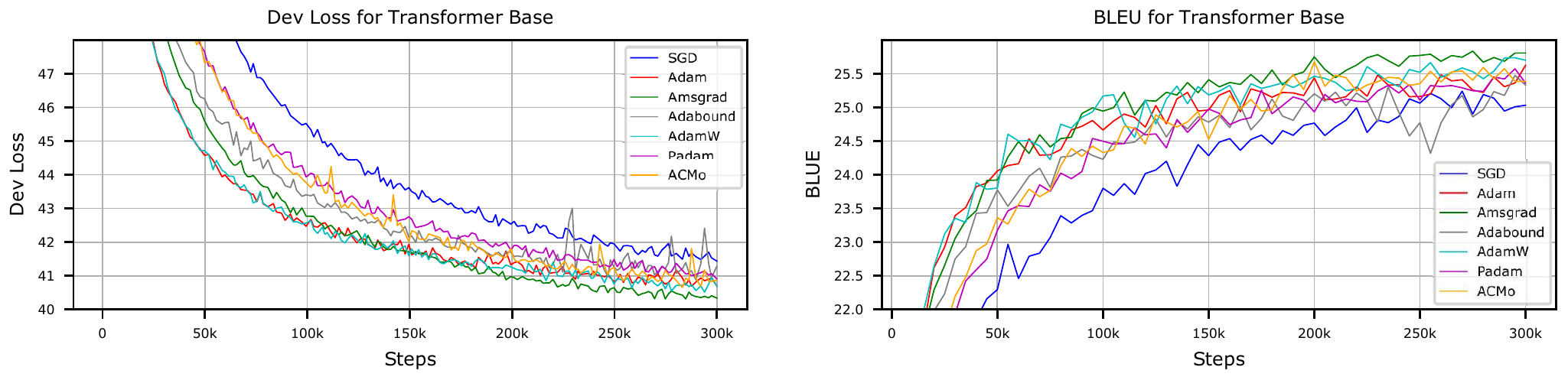}
    \vspace*{-23pt}
    \caption{\small Learning curves of optimizers for Transformer-base on WMT'14 EN-DE machine translation task. (Dev loss \& BLEU score)}
    \vspace*{-15pt}
    \label{fig:nmt_bleu}
\end{figure*}

\textbf{Experiment with Optimal Weight Decay}
We conduct experiments on CIFAR-10 dataset, and find the optimal weight decay for each optimizer which can achieve the best test accuracy. 

\emph{\method obtains the top three results of generalization performance in different architectures with comparable efficiency to SOTA Adam-type optimizers.} As shown in Figure~\ref{fig:optim_weight_decay}, Adam, AdamW and Amsgrad did not perform as well as other optimizers in the plots of test accuracy.
In this weight decay setting,
PAdam, Adabound and \method can even outperform SGD-momentum on VGGNet at the expense of sacrificing some efficiency. 
Nonetheless, they were able to outperform Adam by more than 2 percentage points for the test accuracy. 
The experimental results show that optimizers which converge faster than \method generalize worse.

\subsection{Experiments on Neural Machine Translation Tasks}
Different from image classification tasks, in NLP tasks with attention models, i.e., neural machine translation (NMT), adaptive gradient methods are still mainstream optimizers~\cite{zhang2019adam}. 
This is especially true for Adam type optimizers, which boast a huge advantage over first moment methods in both convergence and generalization in \emph{the fixed learning rate setting}.
To validate the efficiency and generalization of \method, we performed experiments on WMT'14 EN-DE dataset with Transformer~\cite{vaswani2017attention}.

Now we introduce the experimental settings for NMT tasks. 
For preprocessing, sentences were encoded using byte-pair encoding, which has a shared source target vocabulary of about 37000 tokens. Furthermore, we utilized 4 Tesla-V100-PCIE-(16GB) GPUs to train the Transformer base, where we set batch token size as $4096$ per GPU in the training process.
In order to remove the interference of learning tricks, we conducted experiments with fixed step size and gradient clipping. 

\emph{ACMo obtains a comparable convergence and generalization performance to mainstream adaptive gradient methods, and a better performance to SGD-momentum from both convergence and generalization in the fix learning rate setting.}
From Figure~\ref{fig:nmt_bleu}, We observed that even though the descend of ACMo in the early stage of training is somewhat slow, its rapid descend in the middle stage of optimization allows it to outperform some adaptive gradient methods, i.e. Adabound and Padam. Finally, ACMo  yielded DevLoss and BLEU results that are comparable with those in Adam and AdamW, and displayed better results than SGD momentum on the fixed learning rate settings. The above results match the records in other literature~\cite{zhang2019adam}.
\vspace{-5pt}

\section{Conclusion}
\label{sec:conclu}
\vspace{-5pt}
In this paper, we revisited the existing adaptive optimization methods from a novel point of view.
We found that 
the widely used second moments essentially penalize the projection of the current descent direction on previous gradients.
Following such a new idea, we proposed a new method \method.
It removes the second moments and constructs a decent direction by forming acute angles with both current and (approximated) previous gradients.
We analyzed its convergence property in the nonconvex setting, and denote that \method shares the same convergence about gradient norm with SOTA Adam-type optimizers. Last, extensive experiments on CV and NLP tasks validated its efficiency and generalization ability.

\bibliographystyle{unsrt}  
\bibliography{references}  

\begin{thebibliography}{10}

\bibitem{Duchi2011}
John Duchi, Elad Hazan, and Yoram Singer.
\newblock Adaptive subgradient methods for online learning and stochastic
  optimization.
\newblock {\em Journal of Machine Learning Research}, 12(Jul):2121--2159, 2011.

\bibitem{Kingma2014}
Diederik~P Kingma and Jimmy Ba.
\newblock Adam: A method for stochastic optimization.
\newblock {\em arXiv preprint arXiv:1412.6980}, 2014.

\bibitem{wilson2017marginal}
Ashia~C Wilson, Rebecca Roelofs, Mitchell Stern, Nati Srebro, and Benjamin
  Recht.
\newblock The marginal value of adaptive gradient methods in machine learning.
\newblock In {\em Advances in Neural Information Processing Systems}, pages
  4148--4158, 2017.

\bibitem{Nesterov2006Cubic}
Yurii Nesterov and Boris~T Polyak.
\newblock Cubic regularization of newton method and its global performance.
\newblock {\em Mathematical Programming}, 108(1):177--205, 2006.

\bibitem{nesterov2013introductory}
Yurii Nesterov.
\newblock {\em Introductory lectures on convex optimization: A basic course},
  volume~87.
\newblock Springer Science \& Business Media, 2013.

\bibitem{Hinton2012}
Geoffrey Hinton, Nitish Srivastava, and Kevin Swersky.
\newblock Neural networks for machine learning lecture 6a overview of
  mini-batch gradient descent.
\newblock {\em Cited on}, 14:8, 2012.

\bibitem{anil2019memory}
Rohan Anil, Vineet Gupta, Tomer Koren, and Yoram Singer.
\newblock Memory-efficient adaptive optimization for large-scale learning.
\newblock {\em arXiv preprint arXiv:1901.11150}, 2019.

\bibitem{Reddi2019}
Sashank~J Reddi, Satyen Kale, and Sanjiv Kumar.
\newblock On the convergence of adam and beyond.
\newblock {\em arXiv preprint arXiv:1904.09237}, 2019.

\bibitem{zhang2017normalized}
Zijun Zhang, Lin Ma, Zongpeng Li, and Chuan Wu.
\newblock Normalized direction-preserving adam.
\newblock {\em arXiv preprint arXiv:1709.04546}, 2017.

\bibitem{Loshchilov2017}
Ilya Loshchilov and Frank Hutter.
\newblock Fixing weight decay regularization in adam.
\newblock {\em arXiv preprint arXiv:1711.05101}, 2017.

\bibitem{keskar2017improving}
Nitish~Shirish Keskar and Richard Socher.
\newblock Improving generalization performance by switching from adam to sgd.
\newblock {\em arXiv preprint arXiv:1712.07628}, 2017.

\bibitem{Luo2019}
Liangchen Luo, Yuanhao Xiong, Yan Liu, and Xu~Sun.
\newblock Adaptive gradient methods with dynamic bound of learning rate.
\newblock In {\em Proceedings of the 7th International Conference on Learning
  Representations}, New Orleans, Louisiana, May 2019.

\bibitem{Chen2018}
Jinghui Chen and Quanquan Gu.
\newblock Closing the generalization gap of adaptive gradient methods in
  training deep neural networks.
\newblock {\em arXiv preprint arXiv:1806.06763}, 2018.

\bibitem{young1912}
William~Henry Young.
\newblock On classes of summable functions and their fourier series.
\newblock {\em Proceedings of the Royal Society of London. Series A, Containing
  Papers of a Mathematical and Physical Character}, 87(594):225--229, 1912.

\bibitem{chen2018convergence}
Xiangyi Chen, Sijia Liu, Ruoyu Sun, and Mingyi Hong.
\newblock On the convergence of a class of adam-type algorithms for non-convex
  optimization.
\newblock {\em arXiv preprint arXiv:1808.02941}, 2018.

\bibitem{vaswani2017attention}
Ashish Vaswani, Noam Shazeer, Niki Parmar, Jakob Uszkoreit, Llion Jones,
  Aidan~N Gomez, {\L}ukasz Kaiser, and Illia Polosukhin.
\newblock Attention is all you need.
\newblock In {\em Advances in neural information processing systems}, pages
  5998--6008, 2017.

\bibitem{krizhevsky2009learning}
Alex Krizhevsky, Geoffrey Hinton, et~al.
\newblock Learning multiple layers of features from tiny images.
\newblock Technical report, Citeseer, 2009.

\bibitem{Simonyan2014}
Karen Simonyan and Andrew Zisserman.
\newblock Very deep convolutional networks for large-scale image recognition.
\newblock {\em arXiv preprint arXiv:1409.1556}, 2014.

\bibitem{He2016}
Kaiming He, Xiangyu Zhang, Shaoqing Ren, and Jian Sun.
\newblock Deep residual learning for image recognition.
\newblock In {\em Proceedings of the IEEE conference on computer vision and
  pattern recognition}, pages 770--778, 2016.

\bibitem{Huang2017}
Gao Huang, Zhuang Liu, Laurens Van Der~Maaten, and Kilian~Q Weinberger.
\newblock Densely connected convolutional networks.
\newblock In {\em Proceedings of the IEEE conference on computer vision and
  pattern recognition}, pages 4700--4708, 2017.

\bibitem{zhang2019adam}
Jingzhao Zhang, Sai~Praneeth Karimireddy, Andreas Veit, Seungyeon Kim,
  Sashank~J Reddi, Sanjiv Kumar, and Suvrit Sra.
\newblock Why adam beats sgd for attention models.
\newblock {\em arXiv preprint arXiv:1912.03194}, 2019.

\end{thebibliography}

\newpage
\appendix
\section{Some Important Lemmas}
In this section, we show several important lemmas for preparing the proof of our main theorem. 
Besides, we also denote some notations to make the main theorem more clear.
\begin{lemma}
    \label{lem:acutum_bounded}
    If we update $\hat{\rvm}_{t}$ as the Step.5 in Algorithm 1, $\hat{\rvm}_{t}$ satisfies:
    \begin{equation*}
        \begin{split}
            (1-\beta_t)\left\|\rvg_t\right\| \le \left\|\hat{\rvm}_t\right\| \le (1+\beta_t)\left\|\rvg_t\right\|
        \end{split}
    \end{equation*}
    when $0\le \beta_t \le 1$ and $\delta_t \ge 0$.
\end{lemma}
\begin{proof}
    For any given $\hat{\rvm}_{t-1}$ at the Step.5 in Algorithm 1, we have
    \begin{equation*}
        \left\|\hat{\rvm}_t\right\| = \left\|\hat{\rvm}_{t-1}\cdot \frac{\left\|\rvg_t\right\|}{\left\|\hat{\rvm}_{t-1}\right\|+\delta_t}\cdot \beta_t + \rvg_t\right\|.
    \end{equation*}
    With the triangle inequality, we have
    \begin{equation}
        \label{eq:lem_1_triangle_ine}
        \left\|\rvg_t\right\| - \beta_t\cdot\left\|\hat{\rvm}_{t-1}\cdot \frac{\left\|\rvg_t\right\|}{\left\|\hat{\rvm}_{t-1}\right\|+\delta_t}\right\| \le \left\|\hat{\rvm}_{t}\right\| \le \left\|\rvg_t\right\|+\beta_t\cdot\left\|\hat{\rvm}_{t-1}\cdot \frac{\left\|\rvg_t\right\|}{\left\|\hat{\rvm}_{t-1}\right\|+\delta_t}\right\|,
    \end{equation}
    where
    \begin{equation}
        \label{eq:lem_1_cauthy_ine}
        \left\|\hat{\rvm}_{t-1}\cdot \frac{\left\|\rvg_t\right\|}{\left\|\hat{\rvm}_{t-1}\right\|+\delta_t}\right\|\le \left\|\hat{\rvm}_{t-1}\right\|\cdot \frac{\left\|\rvg_t\right\|}{\left\|\hat{\rvm}_{t-1}\right\|+\delta_t}\le \left\|\rvg_t\right\|.
    \end{equation}
    Submitting \Eqref{eq:lem_1_cauthy_ine} back into \Eqref{eq:lem_1_triangle_ine}, we complete the proof.
\end{proof}
\begin{lemma}
    \label{lem:mini_batch_variance_bound}
    According to Assumption~4.3, for any $1\le t\le T$, if the size of mini-batch $\left|\sA_{t}\right|$ in the training process is $B$, we have
    \begin{equation*}
        \left\|\nabla f_{\sA_{t}}(\vtheta_t) - \nabla f(\vtheta_t) \right\|\le \sigma
    \end{equation*}
\end{lemma}
\begin{proof}
    With the triangle inequality, we have
    \begin{equation}
        \begin{split}
            \left\|\nabla f_{\sA_{t}}(\vtheta_t) - \nabla f(\vtheta_t) \right\| \le \frac{1}{B}\sum_{i\in \sA_t}\left\|\nabla f_i(\vtheta_t) - \nabla f(\vtheta_t)\right\|\le \sigma
        \end{split}
    \end{equation}
    completing the proof.
\end{proof}
\begin{lemma}
    \label{lem:mini_batch_norm_bound}
    For each iteration $1\le t\le T$, we have
    \begin{equation}
        \mathbb{E}\left[\left\|\hat{\rvm}_t\right\|^2\right] \le \mathbb{E}\left[\left(1+\beta_t\right)^2\left\|\rvg_t\right\|^2\right]\le \frac{\left(1+\beta_t\right)^2}{2}\left(\sigma^2 + \mathbb{E}\left[\left\|\nabla f(\vtheta_t)\right\|^2\right]\right).
    \end{equation}
\end{lemma}
\begin{proof}
    According to Lemma~\ref{lem:acutum_bounded}, we have
    \begin{equation*}
        \mathbb{E}\left[\left\|\hat{\rvm}_t\right\|^2\right] \le \mathbb{E}\left[\left(1+\beta_t\right)^2\left\|\rvg_t\right\|^2\right]
    \end{equation*}
    Besides, with Assumption~4.3, we then obtain
    \begin{equation}
        \begin{split}
            \mathbb{E}\left[\left\|\rvg_t\right\|^2\right] =& \mathbb{E}\left[\left\|\rvg_t-\nabla f(\vtheta_t) + \nabla f(\vtheta_t)\right\|^2\right]\\
            \le &\frac{1}{2}\left(\mathbb{E}\left[\left\|\rvg_t - \nabla f(\vtheta_t)\right\|^2\right] + \mathbb{E}\left[\left\|\nabla f(\vtheta_t)\right\|^2\right]\right)\\
            = & \frac{1}{2}\sigma^2 + \frac{1}{2}\mathbb{E}\left[\left\|\nabla f(\vtheta_t)\right\|^2\right],
        \end{split}
    \end{equation}
    for the mini-batch gradient $\rvg_t$ to complete the proof.
\end{proof}
\begin{lemma}
    With assumptions proposed in Section 4, for any given $\delta_t \ge \sigma$ and mini-batch size satisfies $\left|\sA_{t}\right| = B$, we have
    \begin{equation*}
        \begin{split}
            \frac{\left\|\rvg_t\right\|}{\left\|\hat{\rvm}_{t-1}\right\| + \delta_t}\le 2+L\alpha_{t-1} + \frac{1}{1-\beta_{t-1}},
        \end{split}
    \end{equation*}
    where $\alpha_t$ is the step size at iteration $t$.
\end{lemma}
\begin{proof}
    With the definition of $\rvg_t$, we have $\rvg_t = \nabla f_{\sA_t}(\vtheta_t)$. We have
    \begin{equation}
        \label{eq:lem_3_trianlge_ine}
        \begin{split}
            \left\|\nabla f_{\sA_t}(\vtheta_t)\right\| = & \left\|\nabla f_{\sA_t}(\vtheta_t) - \nabla f(\vtheta_t) + \nabla f(\vtheta_t) - \nabla f(\vtheta_{t-1})\right.\\
            & \left.+ \nabla f(\vtheta_{t-1}) - \nabla f_{\sA_{t-1}}(\vtheta_{t-1})+\nabla f_{\sA_{t-1}}(\vtheta_{t-1})\right\|\\
            \le & \left\|\nabla f_{\sA_t}(\vtheta_t) - \nabla f(\vtheta_t)\right\| + \left\|\nabla f(\vtheta_t) - \nabla f(\vtheta_{t-1})\right\| \\
            & + \left\|\nabla f(\vtheta_{t-1}) - \nabla f_{\sA_{t-1}}(\vtheta_{t-1})\right\| + \left\|\nabla f_{\sA_{t-1}}(\vtheta_{t-1})\right\|.
        \end{split}
    \end{equation}
    Consider the expectation of $\left\|\rvg_t\right\|$ and~\Eqref{eq:lem_3_trianlge_ine}, we have
    \begin{equation}
        \label{eq:lem_3_des_ine}
        \begin{split}
            \frac{\left\|\rvg_t\right\|}{\left\|\hat{\rvm}_{t-1}\right\| + \delta_t}\le & \underbrace{\frac{\left\|\nabla f_{\sA_t}(\vtheta_t) - \nabla f(\vtheta_t)\right\|}{\left\|\hat{\rvm}_{t-1}\right\| + \delta_t}}_{T_{3.1}} + \underbrace{\frac{\left\|\nabla f(\vtheta_t)- \nabla f(\vtheta_{t-1})\right\|}{\left\|\hat{\rvm}_{t-1}\right\| + \delta_t}}_{T_{3.2}} \\
            & + \underbrace{\frac{\left\|\nabla f (\vtheta_{t-1})- \nabla f_{\sA_{t-1}}(\vtheta_{t-1})\right\|}{\left\|\hat{\rvm}_{t-1}\right\| + \delta_t}}_{T_{3.3}} +\underbrace{\frac{\left\|\nabla f_{\sA_{t-1}(\vtheta_{t-1})}\right\|}{\left\|\hat{\rvm}_{t-1}\right\| + \delta_t}}_{T_{3.4}}
        \end{split}
    \end{equation}

    Since we sample each mini-batch independently, for $T_{2.1}$ in~\Eqref{eq:lem_3_des_ine} we have
    \begin{equation}
        \label{eq:lem_3_T_3_1}
        \begin{split}
            T_{3.1} & = \frac{\left\|\nabla f_{\sA_t}(\vtheta_t) - \nabla f(\vtheta_t)\right\|}{\left\|\hat{\rvm}_{t-1}\right\| + \delta_t} \mathop{\le}^{\mathcircled1} \frac{\sigma}{\left\|\hat{\rvm}_{t-1}\right\| + \delta_t} \mathop{\le}^{\mathcircled2} 1,
        \end{split}
    \end{equation}
    where $\mathcircled1$ and $\mathcircled2$ establish because of Lemma~\ref{lem:mini_batch_variance_bound} and the known condition $\delta_t > \sigma$.

    With Assumption~4.1, for $T_{2.2}$ we have
    \begin{equation}
        \label{eq:lem_3_T_3_2}
        \begin{split}
            T_{3.2} & = \frac{\left\|\nabla f(\vtheta_t)- \nabla f(\vtheta_{t-1})\right\|}{\left\|\hat{\rvm}_{t-1}\right\| + \delta_t} \le  \frac{L\left\|\vtheta_t- \vtheta_{t-1}\right\|}{\left\|\hat{\rvm}_{t-1}\right\| + \delta_t} \le \frac{L\left\|\alpha_{t-1}\hat{\rvm}_{t-1}\right\|}{\left\|\hat{\rvm}_{t-1}\right\| + \delta_t}\le L\alpha_{t-1}.
        \end{split}
    \end{equation}

    Similar to~\Eqref{eq:lem_3_T_3_1}, for $T_{3.3}$, there is $T_{3.3}\le 1$.

    For $T_{3.4}$, we utilize Lemma~\ref{lem:acutum_bounded} to bound as follows:
    \begin{equation}
        \label{eq:lem_3_T_3_4}
        \begin{split}
            T_{3.4} = \frac{\left\|\rvg_{t-1}\right\|}{\left\|\hat{\rvm}_{t-1}\right\| + \delta_t} \le \frac{\left\|\rvg_{t-1}\right\|}{\left\|\hat{\rvm}_{t-1}\right\|} \le \frac{1}{1-\beta_{t-1}}.
        \end{split}
    \end{equation}
    Thus, we can submit~\Eqref{eq:lem_3_T_3_1},~\Eqref{eq:lem_3_T_3_2},~\Eqref{eq:lem_3_T_3_4} back into~\Eqref{eq:lem_3_des_ine} to complete such proof.
\end{proof}
\begin{corollary}
    \label{cor:iter_con_bound}
    If we set $\delta_t>\sigma$, $\alpha_t\le L^{-1}$ for any iteration $1\le t\le T$, we have
    \begin{equation*}
        \frac{\left\|\rvg_t\right\|}{\left\|\hat{\rvm}_{t-1}\right\| + \delta_t}\le 3+\frac{1}{1-\beta_{t-1}}.
    \end{equation*}
    Besides, if $\beta_t\le \frac{1}{50}$ for any $1\le t\le T$  holds, we obtain
    \begin{equation}
        \label{eq:hat_beta_definition}
        \begin{split}
            \hat{\beta}_t:= \beta_t \cdot \frac{\left\|\rvg_t\right\|}{\left\|\hat{\rvm}_{t-1}\right\| + \delta_t} \le \frac{1}{12}.
        \end{split}
    \end{equation} 
\end{corollary}
\begin{lemma}
    \label{lem:constructed_sequence}
    In Algorithm 1, with the hyper-parameters settings as Corollary~\ref{cor:iter_con_bound}, and $\vtheta_0 = \vtheta_1$, for the sequence $\left\{\vtheta^\prime_i\right\}$ which is defined as
    \begin{equation*}
        \vtheta^\prime_i = \vtheta_i+\left(\frac{6}{\sqrt{i+1}}+1\right)\cdot \frac{\hat{\beta}_i}{\sqrt{\frac{i}{i-1}}-\hat{\beta}_i} (\vtheta_i-\vtheta_{i-1}),\ \forall\ i\ge 2,
    \end{equation*}
    the following holds true
    \begin{equation}
        \label{eq:constructed_sequence_itration}
        \begin{split}
            \vtheta_{i+1}^\prime= &\vtheta_i^\prime - \alpha_i\cdot\frac{\sqrt{\frac{1}{i}}+1}{1-\hat{\beta}_i} \cdot \rvg_i -\left(\alpha_i\cdot \frac{\sqrt{\frac{1}{i}}+1}{1-\hat{\beta}_i}\cdot \hat{\beta}_i-\alpha_{i-1}\left(\frac{6}{\sqrt{i+1}}+1\right)\cdot \frac{\hat{\beta}_i}{\sqrt{\frac{i}{i-1}}-\hat{\beta}_i}\right)\cdot \hat{\rvm}_{t-1}\\
            & + \left[\left(\frac{6}{\sqrt{i+2}}+1\right)\cdot \frac{\hat{\beta}_{i+1}}{\sqrt{\frac{i+1}{i}}-\hat{\beta}_{i+1}} - \left(\frac{\sqrt{\frac{1}{i}}+1}{1-\hat{\beta}_i} - 1\right)\right]\cdot \left(\vtheta_{i+1}-\vtheta_i\right)
        \end{split}
    \end{equation}
    when $t\ge 2$. Besides, we have
    \begin{equation}
        \label{eq:constructed_sequence_t2}
        \begin{split}
            \vtheta_2^\prime - \vtheta_1^\prime = -\alpha_1\cdot \frac{\sqrt{2}+1}{1-\hat{\beta_1}} \cdot \rvg_1 + \left[\left(\frac{6}{\sqrt{3}}+1\right)\cdot \frac{\hat{\beta}_2}{\sqrt{2}-\hat{\beta}_2}- \frac{\sqrt{2}+1}{1-\hat{\beta_1}}+1\right](\vtheta_2-\vtheta_1)
        \end{split}
    \end{equation}
\end{lemma}
\begin{proof}
    We first rearrange the coefficients of $\vtheta_i$'s iteration in the following
    \begin{equation}
        \label{eq:coef_rearrange}
        \begin{split}
            \frac{\sqrt{\frac{1}{i}}+1}{1-\hat{\beta}_i} \cdot \left(\vtheta_{i+1}-\vtheta_i\right) = \vtheta_{i+1}+\left(\frac{\sqrt{\frac{1}{i}}+1}{1-\hat{\beta}_i} - 1\right)\left(\vtheta_{i+1}-\vtheta_i\right) - \vtheta_{i}
        \end{split}
    \end{equation}
    According to the update rule of parameters, we obtain
    \begin{equation}
        \label{eq:iteration_rearrange}
        \begin{split}
            \frac{\sqrt{\frac{1}{i}}+1}{1-\hat{\beta}_i} \cdot \left(\vtheta_{i+1}-\vtheta_i\right) &= \frac{\sqrt{\frac{1}{i}}+1}{1-\hat{\beta}_i} \cdot \left(-\alpha_i\hat{\rvm}_{i}\right) = -\alpha_i \cdot \frac{\sqrt{\frac{1}{i}}+1}{1-\hat{\beta}_i} \cdot \left(\rvg_i+\hat{\beta}_i\hat{\rvm}_{i-1}\right)\\   
            = &-\alpha_i \cdot \frac{\sqrt{\frac{1}{i}}+1}{1-\hat{\beta}_i} \cdot \rvg_i + \frac{\alpha_i \cdot \frac{\sqrt{\frac{1}{i}}+1}{1-\hat{\beta}_i}}{\alpha_{i-1}}\cdot \hat{\beta}_i\cdot \left(\vtheta_{i}-\vtheta_{i-1}\right)\\
            = &\left(\frac{6}{\sqrt{i+1}}+1\right)\cdot \frac{\hat{\beta}_i}{\sqrt{\frac{i}{i-1}}-\hat{\beta}_i}\cdot \left(\vtheta_i - \vtheta_{i-1}\right)-\alpha_i \cdot \frac{\sqrt{\frac{1}{i}}+1}{1-\hat{\beta}_i} \cdot \rvg_i\\
            & +\left(\frac{\alpha_i\cdot \frac{\sqrt{\frac{1}{i}}+1}{1-\hat{\beta}_i}\cdot \hat{\beta}_i}{\alpha_{i-1}}-\left(\frac{6}{\sqrt{i+1}}+1\right)\cdot \frac{\hat{\beta}_i}{\sqrt{\frac{i}{i-1}}-\hat{\beta}_i}\right)\cdot \left(\vtheta_i-\vtheta_{i-1}\right)
        \end{split}
    \end{equation}
    Combine~\Eqref{eq:coef_rearrange} and \Eqref{eq:iteration_rearrange}, when $i\ge 2$, we obtain the following inequalities
    \begin{equation}
        \label{eq:temp_ine}
        \begin{split}
            & \vtheta_{i+1}+\left(\frac{\sqrt{\frac{1}{i}}+1}{1-\hat{\beta}_i} - 1\right)\left(\vtheta_{i+1}-\vtheta_i\right) - \vtheta_{i}\\
            =& \left(\frac{6}{\sqrt{i+1}}+1\right)\cdot \frac{\hat{\beta}_i}{\sqrt{\frac{i}{i-1}}-\hat{\beta}_i}\cdot \left(\vtheta_i - \vtheta_{i-1}\right)-\alpha_i \cdot \frac{\sqrt{\frac{1}{i}}+1}{1-\hat{\beta}_i} \cdot \rvg_i\\
            &+\left(\frac{\alpha_i\cdot \frac{\sqrt{\frac{1}{i}}+1}{1-\hat{\beta}_i}\cdot \hat{\beta}_i}{\alpha_{i-1}}-\left(\frac{6}{\sqrt{i+1}}+1\right)\cdot \frac{\hat{\beta}_i}{\sqrt{\frac{i}{i-1}}-\hat{\beta}_i}\right)\cdot \left(\vtheta_i-\vtheta_{i-1}\right)
        \end{split}
    \end{equation}
    Rearrange terms in~\Eqref{eq:temp_ine}, we have
    \begin{equation}
        \label{eq:iteration_equation}
        \begin{split}
            &\vtheta_{i+1} + \left(\frac{\sqrt{\frac{1}{i}}+1}{1-\hat{\beta}_i} - 1\right)\left(\vtheta_{i+1}-\vtheta_i\right)\\
            =& \vtheta_i + \left(\frac{6}{\sqrt{i+1}}+1\right)\cdot \frac{\hat{\beta}_i}{\sqrt{\frac{i}{i-1}}-\hat{\beta}_i}\cdot \left(\vtheta_i - \vtheta_{i-1}\right)\\
            & - \alpha_i\cdot\frac{\sqrt{\frac{1}{i}}+1}{1-\hat{\beta}_i} \cdot \rvg_i -\left(\alpha_i\cdot \frac{\sqrt{\frac{1}{i}}+1}{1-\hat{\beta}_i}\cdot \hat{\beta}_i-\alpha_{i-1}\left(\frac{6}{\sqrt{i+1}}+1\right)\cdot \frac{\hat{\beta}_i}{\sqrt{\frac{i}{i-1}}-\hat{\beta}_i}\right)\cdot \hat{\rvm}_{t-1}
        \end{split}
    \end{equation}
    Thus, if we denote $\vtheta_i^\prime$ as
    \begin{equation*}
        \vtheta^\prime_i = \vtheta_i+\left(\frac{6}{\sqrt{i+1}}+1\right)\cdot \frac{\hat{\beta}_i}{\sqrt{\frac{i}{i-1}}-\hat{\beta}_i} (\vtheta_i-\vtheta_{i-1}),
    \end{equation*}
    we then obtain
    \begin{equation}
        \label{eq:des_equation_ge}
        \begin{split}
            &\vtheta_{i+1}^\prime = \vtheta_{i+1}+\left(\frac{6}{\sqrt{i+2}}+1\right)\cdot \frac{\hat{\beta}_{i+1}}{\sqrt{\frac{i+1}{i}}-\hat{\beta}_{i+1}} (\vtheta_{i+1}-\vtheta_i)\\
            = & \underbrace{\vtheta_{i+1} + \left(\frac{\sqrt{\frac{1}{i}}+1}{1-\hat{\beta}_i} - 1\right)\left(\vtheta_{i+1}-\vtheta_i\right)}_{T_1}\\
            & + \left[\left(\frac{6}{\sqrt{i+2}}+1\right)\cdot \frac{\hat{\beta}_{i+1}}{\sqrt{\frac{i+1}{i}}-\hat{\beta}_{i+1}} - \left(\frac{\sqrt{\frac{1}{i}}+1}{1-\hat{\beta}_i} - 1\right)\right]\cdot \left(\vtheta_{i+1}-\vtheta_i\right).
        \end{split}
    \end{equation}
    Introducing \Eqref{eq:iteration_equation} to $T_{1}$ of \Eqref{eq:des_equation_ge}, we have
    \begin{equation}
        \label{eq:constructed_sequence_final}
        \begin{split}
            \vtheta_{i+1}^\prime =& \vtheta_i + \left(\frac{6}{\sqrt{i+1}}+1\right)\cdot \frac{\hat{\beta}_i}{\sqrt{\frac{i}{i-1}}-\hat{\beta}_i}\cdot \left(\vtheta_i - \vtheta_{i-1}\right)\\
            & - \alpha_i\cdot \frac{\sqrt{\frac{1}{i}}+1}{1-\hat{\beta}_i} \cdot \rvg_i -\left(\alpha_i\cdot \frac{\sqrt{\frac{1}{i}}+1}{1-\hat{\beta}_i}\cdot \hat{\beta}_i-\alpha_{i-1}\left(\frac{6}{\sqrt{i+1}}+1\right)\cdot \frac{\hat{\beta}_i}{\sqrt{\frac{i}{i-1}}-\hat{\beta}_i}\right)\cdot \hat{\rvm}_{i-1}\\
            = &\vtheta_i^\prime - \alpha_i\cdot\frac{\sqrt{\frac{1}{i}}+1}{1-\hat{\beta}_i} \cdot \rvg_i -\left(\alpha_i\cdot \frac{\sqrt{\frac{1}{i}}+1}{1-\hat{\beta}_i}\cdot \hat{\beta}_i-\alpha_{i-1}\left(\frac{6}{\sqrt{i+1}}+1\right)\cdot \frac{\hat{\beta}_i}{\sqrt{\frac{i}{i-1}}-\hat{\beta}_i}\right)\cdot \hat{\rvm}_{t-1}\\
            & + \left[\left(\frac{6}{\sqrt{i+2}}+1\right)\cdot \frac{\hat{\beta}_{i+1}}{\sqrt{\frac{i+1}{i}}-\hat{\beta}_{i+1}} - \left(\frac{\sqrt{\frac{1}{i}}+1}{1-\hat{\beta}_i} - 1\right)\right]\cdot \left(\vtheta_{i+1}-\vtheta_i\right).
        \end{split}
    \end{equation}
    Additionally, when $i=1$ and $\vtheta_1^\prime = \vtheta_1$, we have
    \begin{equation}
        \label{eq:constructed_sequence_final}
        \begin{split}
            \vtheta_2^\prime - \vtheta_1^\prime =& \vtheta_2 + \left(\frac{6}{\sqrt{3}}+1\right)\cdot \frac{\hat{\beta}_2}{\sqrt{2}-\hat{\beta}_2} (\vtheta_2-\vtheta_1)-\vtheta_1\\
            = & \vtheta_2 + \frac{\sqrt{2}+1}{1-\hat{\beta_1}} \cdot\left(\vtheta_2-\vtheta_1\right) + \left[\left(\frac{6}{\sqrt{3}}+1\right)\cdot \frac{\hat{\beta}_2}{\sqrt{2}-\hat{\beta}_2}- \frac{\sqrt{2}+1}{1-\hat{\beta_1}}\right] (\vtheta_2-\vtheta_1)-\vtheta_1\\
            =& -\alpha_1\cdot \frac{\sqrt{2}+1}{1-\hat{\beta_1}} \cdot \rvg_1 + \left[\left(\frac{6}{\sqrt{3}}+1\right)\cdot \frac{\hat{\beta}_2}{\sqrt{2}-\hat{\beta}_2}- \frac{\sqrt{2}+1}{1-\hat{\beta_1}}+1\right](\vtheta_2-\vtheta_1)
        \end{split}
    \end{equation}
    to complete our proof.
\end{proof}

\section{Proofs in Section 4}
In this section, we provide the proofs for theorems and corollaries in Section 4.
\subsection{Proof of Theorem 4.1}
\label{prf:gradient_norm_expectation_bound}
\begin{proof}
    Combining Assumption~4.1 and Lemma~\ref{lem:constructed_sequence}, for the sequence $\left\{\vtheta_t^\prime\right\}$ we have
    \begin{equation}
        \label{eq:z_lipschitz_continuty}
        \begin{split}
             &\mathbb{E}\left[f(\vtheta^\prime_{t+1})-f(\vtheta^\prime_1)\right] \\
            = &\mathbb{E}\left[\sum_{i=1}^t \left(f(\vtheta^\prime_{i+1}) - f(\vtheta^\prime_i)\right)\right]\\
            \le &\mathbb{E}\left[\sum_{i=1}^t  \left(\nabla f(\vtheta^\prime_i)^T\left(\vtheta\prime_{i+1}-\vtheta^\prime_i\right)+\frac{L}{2}\left\|\vtheta^\prime_{i+1}-\vtheta^\prime_i\right\|^2\right)\right]\\
            =&\mathbb{E}\left[\sum_{i=1}^t  \left(\left(\nabla f(\vtheta^\prime_i)- \nabla f(\vtheta_i)\right)^T\left(\vtheta^\prime_{i+1} - \vtheta^\prime_i\right)  + \nabla f(\vtheta_i)^T\left(\vtheta^\prime_{i+1} - \vtheta^\prime_i\right) + \frac{L}{2}\left\|\vtheta^\prime_{i+1} - \vtheta^\prime_i\right\|^2\right)\right]\\
            \le & \mathbb{E}\left[\sum_{i=1}^t\left(\frac{1}{2L}\left\|\nabla f(\vtheta_i^\prime) - \nabla f(\vtheta_i))\right\|^2 + \nabla f(\vtheta_i)^T\left(\vtheta^\prime_{i+1} - \vtheta^\prime_i\right) + L\left\|\vtheta_{i+1}^\prime - \vtheta_i^\prime\right\|^2\right)\right]\\
            \le & \underbrace{\mathbb{E}\left[\sum_{i=1}^t \frac{L}{2}\left\|\vtheta_i^\prime - \vtheta_i\right\|^2\right]}_{T_1} + \underbrace{\mathbb{E}\left[\sum_{i=1}^t \nabla f(\vtheta_i)^T\left(\vtheta_{i+1}^\prime - \vtheta_i\right)\right]}_{T_2} + \underbrace{\mathbb{E}\left[\sum_{i=1}^t L\left\|\vtheta_{i+1}^\prime - \vtheta_{i}^\prime\right\|^2\right]}_{T_3}
        \end{split}
    \end{equation}

    \emph{For $T_1$ in \Eqref{eq:z_lipschitz_continuty}:} we have
    \begin{equation}
        \label{eq:t1_for_z_lipschitz_continuty}
        \begin{split}
            & \mathbb{E}\left[\sum_{i=1}^t \frac{L}{2}\left\|\vtheta_i^\prime - \vtheta_i\right\|^2\right]=\mathbb{E}\left[\sum_{i=2}^t \frac{L}{2}\cdot \left(\frac{6}{\sqrt{i+1}}+1\right)^2 \cdot \left(\frac{\hat{\beta}_i}{\sqrt{\frac{i}{i-1}}-\hat{\beta}_i}\right)^2\cdot\left\|\vtheta_i - \vtheta_{i-1}\right\|^2\right]\\
            \mathop{\le}^{\mathcircled1} & \mathbb{E}\left[\sum_{i=2}^t \frac{L}{2}\left\|\vtheta_{i}-\vtheta_{i-1}\right\|^2\right] = \mathbb{E}\left[\sum_{i=2}^t \frac{L}{2}\cdot \alpha_{i-1}^2\cdot \left\|\hat{\rvm}_{i-1}\right\|^2\right] \mathop{\le}^{\mathcircled2}\mathbb{E}\left[\sum_{i=2}^t \frac{L}{2}\cdot \alpha_{i-1}^2 \cdot \left(1+\beta_{i-1}\right)^2 \cdot \left\|\rvg_{i-1}\right\|^2\right]\\
            \le & \mathbb{E}\left[\sum_{i=2}^t \frac{L}{2}\cdot \alpha_{i-1}^2 \cdot \left(1+\beta_{i-1}\right)^2 \cdot \left\|\rvg_{i-1} - \nabla f(\vtheta_{i-1}) + \nabla f(\vtheta_{i-1})\right\|^2\right]\\
            \le & \mathbb{E}\left[\sum_{i=1}^t \frac{L}{4}\cdot \alpha_{i}^2 \cdot \left(1+\beta_{i}\right)^2 \cdot \sigma^2\right] + \mathbb{E}\left[\sum_{i=1}^t \frac{L}{4}\cdot \alpha_{i}^2 \cdot \left(1+\beta_{i}\right)^2 \cdot \left\|\nabla f(\vtheta_{i})\right\|^2\right]\\
            \le & \frac{L\sigma^2\alpha_0^2}{2}\cdot\left(1+\ln(t)\right) + \mathbb{E}\left[\sum_{i=1}^t \frac{L}{2}\cdot \alpha_{i}^2 \cdot \left\|\nabla f(\vtheta_{i})\right\|^2\right]
        \end{split}
    \end{equation}
    where $\mathcircled1$ and $\mathcircled2$ establish becuase of Corollary~\ref{cor:iter_con_bound} and Lemma~\ref{lem:acutum_bounded}. 

    \emph{For $T_2$ in \Eqref{eq:z_lipschitz_continuty}:} it satisfies
    \begin{equation}
        \label{eq:t2_for_z_lipschitz_continuty}
        \begin{split}
        &\mathbb{E}\left[\sum_{i=1}^t \nabla f(\vtheta_i)^T\left(\vtheta_{i+1}^\prime - \vtheta_i^\prime\right)\right]\\
         = & \underbrace{\mathbb{E}\left[\sum_{i=2}^t \nabla f(\vtheta_i)^T\left(\left(\frac{6}{\sqrt{i+1}}+1\right)\cdot\frac{\alpha_{i-1}\hat{\beta}_i}{\sqrt{\frac{i}{i-1}}-\hat{\beta}_i} - \left(\sqrt{\frac{1}{i}}+1\right)\cdot\frac{\alpha_i\hat{\beta}_i}{1-\hat{\beta}_i}\right)\hat{m}_{i-1}\right]}_{T_{2.1}}\\
        & - \underbrace{\mathbb{E}\left[\sum_{i=1}^t \nabla f(\vtheta_i)^T\left(\left(\sqrt{\frac{1}{i}}+1\right)\cdot\frac{\alpha_i}{1-\hat{\beta}_i}\right) \rvg_i\right]}_{T_{2.2}}\\
        &+\underbrace{\mathbb{E}\left[\sum_{i=1}^t \nabla f(\vtheta_i)^T\left(\left(\frac{6}{\sqrt{i+2}}+1\right)\cdot\frac{\hat{\beta}_{i+1}}{\sqrt{\frac{i+1}{i}}-\hat{\beta}_{i+1}} -\left(\sqrt{\frac{1}{i}}+1\right)\cdot\frac{1}{1-\hat{\beta}_i} +1 \right)\left(\vtheta_i - \vtheta_{i-1}\right)\right]}_{T_{2.3}}
        \end{split}
    \end{equation}
    $T_{2.1}$ can be reformulated as:
    \begin{equation}
        \label{eq:T_2_1_refor}
        \begin{split}
            T_{2.1} =& \mathbb{E}\left[\sum_{i=2}^t \nabla f(\vtheta_i)^T\left(\left(\frac{6}{\sqrt{i+1}}+1\right)\cdot\frac{\alpha_{i-1}\hat{\beta}_i}{\sqrt{\frac{i}{i-1}}-\hat{\beta}_i} - \left(\sqrt{\frac{1}{i}}+1\right)\cdot\frac{\alpha_i\hat{\beta}_i}{1-\hat{\beta}_i}\right)\hat{m}_{i-1}\right]\\
            \mathop{\le}^{\mathcircled1} & \mathbb{E}\left[\sum_{i=2}^t \left[\left(\frac{6}{\sqrt{i+1}}+1\right)\cdot\frac{\alpha_{i-1}\hat{\beta}_i}{\sqrt{\frac{i}{i-1}}-\hat{\beta}_i} - \left(\sqrt{\frac{1}{i}}+1\right)\cdot\frac{\alpha_i\hat{\beta}_i}{1-\hat{\beta}_i}\right]\cdot \left\|\nabla f(\vtheta_i)\right\|\cdot\left\|\hat{m}_{i-1}\right\|\right] \\
            \mathop{\le}^{\mathcircled2} & \mathbb{E}\left[\sum_{i=2}^t \left[\left(\frac{6}{\sqrt{i+1}}+1\right)\cdot\frac{\alpha_{i-1}\hat{\beta}_i}{\sqrt{\frac{i}{i-1}}-\hat{\beta}_i} - \left(\sqrt{\frac{1}{i}}+1\right)\cdot\frac{\alpha_i\hat{\beta}_i}{1-\hat{\beta}_i}\right]\cdot G^2\right]\\
            = & \mathbb{E}\left[\left(\frac{6}{\sqrt{3}}+1\right)\cdot \frac{\alpha_{1}\hat{\beta}_2}{\sqrt{2}-\hat{\beta}_2}\cdot G^2\right] - \mathbb{E}\left[\left(\sqrt{\frac{1}{t}}+1\right)\cdot\frac{\alpha_t\hat{\beta}_t}{1-\hat{\beta}_t}\cdot G^2\right]\\
            & + \underbrace{\mathbb{E}\left[\sum_{i=2}^{t-1}\left[\left(\frac{6}{\sqrt{i+2}}+1\right)\cdot\frac{\alpha_i\hat{\beta}_{i+1}}{\sqrt{\frac{i+1}{i}}-\hat{\beta}_{i+1}} - \left(\frac{1}{\sqrt{i}}+1\right)\cdot \frac{\alpha_i\hat{\beta}_i}{1-\hat{\beta}_i}\right]\cdot G^2\right]}_{T_{2.1.1}}
        \end{split}
    \end{equation}
    where $\mathcircled1$ and $\mathcircled2$ establish becuase of Corollary~\ref{cor:T_2_1_no_negative_property} and gradient norm bounded assumption when $\alpha = \frac{\alpha_0}{\sqrt{t}}$. With the same selection of hyper-parameters, i.e., $\alpha_t$, we have $T_{2.1.1}$ satisfies
    \begin{equation}
        \label{eq:lemma_10_app}
        \begin{split}
            \left[\left(\frac{6}{\sqrt{i+2}}+1\right)\cdot\frac{\alpha_i\hat{\beta}_{i+1}}{\sqrt{\frac{i+1}{i}}-\hat{\beta}_{i+1}} - \left(\frac{1}{\sqrt{i}}+1\right)\cdot \frac{\alpha_i\hat{\beta}_i}{1-\hat{\beta}_i}\right]\le \frac{\alpha_0}{i}
        \end{split}
    \end{equation}
    due to~Lemma~\ref{lem:T_2_1_telescoping_sum_property}. Thus, submitting~\Eqref{eq:lemma_10_app} back into~\Eqref{eq:T_2_1_refor}, we have
    \begin{equation}
        \label{eq:T_2_1_final}
        \begin{split}
            T_{2.1}\le &\mathbb{E}\left[\left(\frac{6}{\sqrt{3}}+1\right)\cdot \frac{\alpha_{1}\hat{\beta}_2}{\sqrt{2}-\hat{\beta}_2}\cdot G^2\right] - \mathbb{E}\left[\left(\sqrt{\frac{1}{t}}+1\right)\cdot\frac{\alpha_t\hat{\beta}_t}{1-\hat{\beta}_t}\cdot G^2\right] + \mathbb{E}\left[\sum_{i=2}^{t-1}\frac{\alpha_0}{i}\right]\\
            \mathop{\le}^{\mathcircled1} & \mathbb{E}\left[\sum_{i=1}^{t-1}\frac{\alpha_0}{i}\right] \le \alpha_0\left(\ln(t)+1\right).
        \end{split}
    \end{equation}
    where we have $\mathcircled1$ becuase $\hat{\beta}_i\le \frac{1}{12}$ got from Corollary~\ref{cor:iter_con_bound}. Then we investigate $T_{2.3}$ which can be reformulated as
    \begin{equation}
        \label{eq:T_2_3_refor}
        \begin{split}
            T_{2.3} =&\mathbb{E}\left[\sum_{i=1}^t \nabla f(\vtheta_i)^T\left(\left(\frac{6}{\sqrt{i+2}}+1\right)\cdot\frac{\hat{\beta}_{i+1}}{\sqrt{\frac{i+1}{i}}-\hat{\beta}_{i+1}} -\left(\sqrt{\frac{1}{i}}+1\right)\cdot\frac{1}{1-\hat{\beta}_i} +1 \right)\left(\vtheta_i - \vtheta_{i-1}\right)\right]\\
            = &\mathbb{E}\left[\sum_{i=1}^t \nabla f(\vtheta_i)^T\left(\left(\frac{6}{\sqrt{i+2}}+1\right)\cdot\frac{\hat{\beta}_{i+1}}{\sqrt{\frac{i+1}{i}}-\hat{\beta}_{i+1}} -\left(\sqrt{\frac{1}{i}}+1\right)\cdot\frac{1}{1-\hat{\beta}_i} +1 \right)\cdot \left(-\alpha_{i-1}\hat{\rvm}_{i-1}\right)\right]\\
            \mathop{\le}^{\mathcircled1} & \underbrace{\mathbb{E}\left[\sum_{i=1}^t\frac{2\left(1-\hat{\beta}_i\right)}{\sqrt{\frac{1}{i}}+1}\cdot \frac{\alpha_{i-1}^2}{\alpha_i}\left|\left(\frac{6}{\sqrt{i+2}}+1\right)\cdot\frac{\hat{\beta}_{i+1}}{\sqrt{\frac{i+1}{i}}-\hat{\beta}_{i+1}} -\left(\sqrt{\frac{1}{i}}+1\right)\cdot\frac{1}{1-\hat{\beta}_i} +1 \right|^2\cdot \left\|\hat{\rvm}_{i-1}\right\|^2\right]}_{T_{2.3.1}}\\
            &+ \mathbb{E}\left[\sum_{i=1}^t\frac{1}{8}\cdot \left(\sqrt{\frac{1}{i}}+1\right)\cdot \frac{\alpha_i}{1-\hat{\beta}_i}\cdot\left\|\nabla f(\vtheta_i)\right\|^2\right],
        \end{split}
    \end{equation}
    where we have $\mathcircled1$ becuase of the Cauthy-Schwarz inequality. With the hyper-parameters selection, for any $i\ge 1$, we have
    \begin{equation}
        \label{eq:T_2_3_constant_upper_bound}
        \begin{split}
            & \left|\left(\frac{6}{\sqrt{i+2}}+1\right)\cdot\frac{\hat{\beta}_{i+1}}{\sqrt{\frac{i+1}{i}}-\hat{\beta}_{i+1}} -\left(\sqrt{\frac{1}{i}}+1\right)\cdot\frac{1}{1-\hat{\beta}_i} +1 \right|\\
            \le & \left|\left(\frac{6}{\sqrt{i+2}}+1\right)\cdot\frac{\hat{\beta}_{i+1}}{\sqrt{\frac{i+1}{i}}-\hat{\beta}_{i+1}}\right|+ \left|\left(\sqrt{\frac{1}{i}}+1\right)\cdot\frac{1}{1-\hat{\beta}_i}\right| + 1 \le 5
        \end{split}
    \end{equation}
    Then, submitting~\Eqref{eq:T_2_3_constant_upper_bound} back into $T_{2.3.1}$ of~\Eqref{eq:T_2_3_refor}, we can obtain
    \begin{equation}
        \label{eq:T_2_3_1_relaxiation}
        \begin{split}
        T_{2.3.1}\le &\mathbb{E}\left[\sum_{i=1}^t \frac{2\left(1-\hat{\beta}_i\right)}{\sqrt{\frac{1}{i}}+1}\cdot \frac{\alpha_{i-1}^2}{\alpha_i}\cdot 5 \cdot \left|\left(\frac{6}{\sqrt{i+2}}+1\right)\cdot\frac{\hat{\beta}_{i+1}}{\sqrt{\frac{i+1}{i}}-\hat{\beta}_{i+1}}\right.\right.\\ 
        &\left.\left.-\left(\sqrt{\frac{1}{i}}+1\right)\cdot\frac{1}{1-\hat{\beta}_i} +1 \right|\cdot \left\|\hat{\rvm}_{i-1}\right\|^2\right]\\
        \mathop{\le}^{\mathcircled1} & \mathbb{E}\left[\sum_{i=1}^t\frac{20\alpha_0\left(1-\hat{\beta}_i\right)}{\sqrt{\frac{1}{i}}+1}\cdot \sqrt{\frac{1}{i}}\cdot \left|\left(\frac{6}{\sqrt{i+2}}+1\right)\cdot\frac{\hat{\beta}_{i+1}}{\sqrt{\frac{i+1}{i}}-\hat{\beta}_{i+1}}-\left(\sqrt{\frac{1}{i}}+1\right)\cdot\frac{1}{1-\hat{\beta}_i} +1 \right|\cdot G^2 \right]\\
        \mathop{\le}^{\mathcircled2} &\mathbb{E}\left[\sum_{i=1}^t\frac{20\alpha_0\left(1-\hat{\beta}_i\right)}{\sqrt{\frac{1}{i}}+1}\cdot \sqrt{\frac{1}{i}}\cdot \left[\left(\sqrt{\frac{1}{i}}+1\right)\cdot\frac{1}{1-\hat{\beta}_i} - \left(\frac{6}{\sqrt{i+2}}+1\right)\cdot\frac{\hat{\beta}_{i+1}}{\sqrt{\frac{i+1}{i}}-\hat{\beta}_{i+1}}-1 \right]\cdot G^2\right]\\
        \le & \mathbb{E}\left[10\alpha_0G^2\sum_{i=1}^t\sqrt{\frac{1}{i}}\cdot \left[\left(\sqrt{\frac{1}{i}}+1\right)\cdot\frac{1}{1-\hat{\beta}_i} - \left(\frac{6}{\sqrt{i+2}}+1\right)\cdot\frac{\hat{\beta}_{i+1}}{\sqrt{\frac{i+1}{i}}-\hat{\beta}_{i+1}}-1 \right]\right]
        \end{split}
    \end{equation}
    where $\mathcircled1$ and $\mathcircled2$ establish becuase of $\alpha_i = \frac{\alpha_0}{\sqrt{i}}$, the gradient norm bounded assumption and Corollary~\ref{cor:T_2_3_no_negative_property}. Then we can rearrange terms of \Eqref{eq:T_2_3_1_relaxiation} as follows
    \begin{equation}
        \label{eq:T_2_3_1_rearr}
        \begin{split}
            T_{2.3.1}\le &10\alpha_0G^2\cdot \mathbb{E}\left[\frac{2}{1-\hat{\beta}_i} - 1\right] +  10\alpha_0G^2\cdot \mathbb{E}\left[\sum_{i=1}^{t-1} \sqrt{\frac{1}{i+1}}\cdot \left(\frac{\sqrt{\frac{1}{i+1}}+1}{1-\hat{\beta}_{i+1}}-1\right) \right.\\
            & \left.-\sqrt{\frac{1}{i}}\cdot\left(\frac{6}{\sqrt{i+2}}+1\right)\cdot\frac{\hat{\beta}_{i+1}}{\sqrt{\frac{i+1}{i}}-\hat{\beta}_{i+1}}\right]\\
            \le &10\alpha_0G^2\cdot \mathbb{E}\left[\frac{2}{1-\hat{\beta}_i} - 1\right] +  10\alpha_0G^2\cdot \mathbb{E}\left[\sum_{i=1}^{t-1}\sqrt{\frac{1}{i}}\cdot\left(\frac{\sqrt{\frac{1}{i+1}}+1}{1-\hat{\beta}_{i+1}}-1 \right.\right.\\
            & \left.\left.-\left(\frac{6}{\sqrt{i+2}}+1\right)\cdot\frac{\hat{\beta}_{i+1}}{\sqrt{\frac{i+1}{i}}-\hat{\beta}_{i+1}} \right)\right]\\
            \mathop{\le}^{\mathcircled1} & 10\alpha_0G^2\cdot \mathbb{E}\left[\frac{2}{1-\hat{\beta}_i} - 1\right] + 10\alpha_0G^2\cdot \mathbb{E}\left[\sum_{i=1}^{t-1}\sqrt{\frac{1}{i}}\cdot \frac{6}{\sqrt{i+2}}\right]\\
            \le & 10\alpha_0G^2\cdot(7+6\ln(t))
        \end{split}
    \end{equation}
    where we have $\mathcircled1$ according to Lemma~\ref{lem:T_2_3_telescoping_sum_property}. Then, we submit \Eqref{eq:T_2_3_1_rearr} back to \Eqref{eq:T_2_3_refor} to have
    \begin{equation}
        \label{eq:T_2_3_final}
        \begin{split}
            T_{2.3}\le \mathbb{E}\left[\frac{1}{8}\cdot \left(\sqrt{\frac{1}{i}}+1\right)\cdot \frac{\alpha_i}{1-\hat{\beta}_i}\cdot \left\|\nabla f(\vtheta_i)\right\|^2\right] + 10\alpha_0G^2\cdot(7+6\ln(t)).
        \end{split}
    \end{equation}
    Thus, we submit \Eqref{eq:T_2_1_final} and \Eqref{eq:T_2_3_final} back into \Eqref{eq:t2_for_z_lipschitz_continuty} to have
    \begin{equation}
        \label{eq:t2_final}
        \begin{split}
            &\mathbb{E}\left[\sum_{i=1}^t \nabla f(\vtheta_i)^T\left(\vtheta_{i+1}^\prime - \vtheta_i\right)\right]\\
            \le &\alpha_0\left(\ln(t)+1\right) - \mathbb{E}\left[\sum_{i=1}^t \nabla f(\vtheta_i)^T\left(\left(\sqrt{\frac{1}{i}}+1\right)\cdot \frac{\alpha_i}{1-\hat{\beta}_i}\right)\rvg_i\right]\\
            &+\mathbb{E}\left[\frac{1}{8}\cdot \left(\sqrt{\frac{1}{i}}+1\right)\cdot \frac{\alpha_i}{1-\hat{\beta}_i}\cdot \left\|\nabla f(\vtheta_i)\right\|^2\right] + 10\alpha_0G^2\cdot(7+6\ln(t))\\
            \le & \left(60G^2+1\right)\alpha_0\ln(t) + \left(70G^2+1\right)\alpha_0- \mathbb{E}\left[\frac{7}{8}\sum_{i=1}^t \alpha_i\left\|\nabla f(\vtheta_i)\right\|^2\right]
        \end{split}
    \end{equation}

    \emph{For $T_3$ in \Eqref{eq:z_lipschitz_continuty}:} we obtain that
    \begin{equation}
        \label{eq:t3_for_z_lipschitz_continuty}
        \begin{split}
            & \mathbb{E}\left[\sum_{i=1}^t L\left\|\vtheta_{i+1}^\prime - \vtheta_{i}^\prime\right\|^2\right]\\
            = &\mathbb{E}\left[\sum_{i=2}^t\left\|- \alpha_i\cdot\frac{\sqrt{\frac{1}{i}}+1}{1-\hat{\beta}_i} \cdot \rvg_i -\left(\alpha_i\cdot \frac{\sqrt{\frac{1}{i}}+1}{1-\hat{\beta}_i}\cdot \hat{\beta}_i-\alpha_{i-1}\left(\frac{6}{\sqrt{i+1}}+1\right)\cdot \frac{\hat{\beta}_i}{\sqrt{\frac{i}{i-1}}-\hat{\beta}_i}\right)\cdot \hat{\rvm}_{i-1} \right.\right.\\
            &\left.\left.+ \left[\left(\frac{6}{\sqrt{i+2}}+1\right)\cdot \frac{\hat{\beta}_{i+1}}{\sqrt{\frac{i+1}{i}}-\hat{\beta}_{i+1}} - \left(\frac{\sqrt{\frac{1}{i}}+1}{1-\hat{\beta}_i} - 1\right)\right]\cdot \left(\vtheta_{i+1}-\vtheta_i\right)\right\|^2\right]\\
            & + \mathbb{E}\left[\left\|-\alpha_1\cdot \frac{\sqrt{2}+1}{1-\hat{\beta_1}} \cdot \rvg_1 + \left[\left(\frac{6}{\sqrt{3}}+1\right)\cdot \frac{\hat{\beta}_2}{\sqrt{2}-\hat{\beta}_2}- \frac{\sqrt{2}+1}{1-\hat{\beta_1}}+1\right](\vtheta_2-\vtheta_1)\right\|^2\right]\\
            \le & \underbrace{\mathbb{E}\left[\sum_{i=1}^t 3\cdot\left\|\alpha_i\cdot\frac{\sqrt{\frac{1}{i}}+1}{1-\hat{\beta}_i} \cdot \rvg_i\right\|^2\right]}_{T_{3.1}}\\
            & + \underbrace{\sum_{i=1}^t \mathbb{E}\left[3\cdot\left\|\left[\left(\frac{6}{\sqrt{i+2}}+1\right)\cdot \frac{\hat{\beta}_{i+1}}{\sqrt{\frac{i+1}{i}}-\hat{\beta}_{i+1}} - \left(\frac{\sqrt{\frac{1}{i}}+1}{1-\hat{\beta}_i} - 1\right)\right]\cdot \left(\vtheta_{i+1}-\vtheta_i\right)\right\|^2\right]}_{T_{3.2}}\\
            & + \underbrace{\mathbb{E}\left[\sum_{i=2}^t 3\cdot\left\|\left(\left(\frac{6}{\sqrt{i+1}}+1\right)\cdot\frac{\alpha_{i-1}\hat{\beta}_i}{\sqrt{\frac{i}{i-1}}-\hat{\beta}_i} - \left(\sqrt{\frac{1}{i}}+1\right)\cdot\frac{\alpha_i\hat{\beta}_i}{1-\hat{\beta}_i}\right)\hat{m}_{i-1}\right\|^2\right]}_{T_{3.3}},
        \end{split}
    \end{equation}
    according to the fact $\left\|p+q+r\right\|^2\le 3\left(\left\|p\right\|^2 + \left\|q\right\|^2 + \left\|r\right\|^2\right),\ \forall\ p,q,r\in \mathbb{R}^d$. Thus, for $T_{3.1}$ of \Eqref{eq:t3_for_z_lipschitz_continuty}, we find
    \begin{equation}
        \label{eq:T_3_1_final}
        \begin{split}
            T_{3.1}=& \mathbb{E}\left[\sum_{i=1}^t 3\cdot\left\|\alpha_i\cdot\frac{\sqrt{\frac{1}{i}}+1}{1-\hat{\beta}_i} \cdot \rvg_i\right\|^2\right] \le \mathbb{E}\left[\sum_{i=1}^t 6\alpha_i^2\cdot \left\|\rvg_i\right\|^2\right]\\
            = &\mathbb{E}\left[\sum_{i=1}^t 6\alpha_i^2\cdot \left\|\rvg_i-\nabla f(\vtheta_i) + \nabla f(\vtheta_i)\right\|^2\right]\\
            \le & \sum_{i=1}^t 3\alpha_i^2\sigma^2 + \mathbb{E}\left[\sum_{i=1}^t 3\alpha_i^2\cdot\left\|\nabla f(\vtheta_i)\right\|^2\right] \le 3\sigma^2\alpha_0^2\left(\ln(t)+1\right)+\mathbb{E}\left[\sum_{i=1}^t 3\alpha_i^2\cdot\left\|\nabla f(\vtheta_i)\right\|^2\right]
        \end{split}
    \end{equation}
    Then, we provide the upper bound of $T_{3.2}$ as follows
    \begin{equation}
        \label{eq:T_3_2_refor}
        \begin{split}
            T_{3.2} =& \mathbb{E}\left[\sum_{i=1}^t 3\cdot\left\|\left[\left(\frac{6}{\sqrt{i+2}}+1\right)\cdot \frac{\hat{\beta}_{i+1}}{\sqrt{\frac{i+1}{i}}-\hat{\beta}_{i+1}} - \left(\frac{\sqrt{\frac{1}{i}}+1}{1-\hat{\beta}_i} - 1\right)\right]\cdot \left(\vtheta_{i+1}-\vtheta_i\right)\right\|^2\right]\\
            =& \mathbb{E}\left[\sum_{i=1}^t 3\alpha_i^2\cdot \left|\left(\frac{6}{\sqrt{i+2}}+1\right)\cdot\frac{\hat{\beta}_{i+1}}{\sqrt{\frac{i+1}{i}}-\hat{\beta}_{i+1}}-\frac{\sqrt{\frac{1}{i}}+1}{1-\hat{\beta}_i} +1 \right|^2\cdot \left\|\hat{\rvm}_i\right\|^2\right]\\
            \mathop{\le}^{\mathcircled1} & \mathbb{E}\left[\sum_{i=1}^t 15\alpha_i^2\cdot \left|\left(\frac{6}{\sqrt{i+2}}+1\right)\cdot\frac{\hat{\beta}_{i+1}}{\sqrt{\frac{i+1}{i}}-\hat{\beta}_{i+1}}-\frac{\sqrt{\frac{1}{i}}+1}{1-\hat{\beta}_i} +1 \right| \cdot \left\|\hat{\rvm}_{i}\right\|^2\right]\\
            \mathop{\le}^{\mathcircled2} &\underbrace{\mathbb{E}\left[\sum_{i=1}^t \frac{15\alpha_i^2(1+\beta_i)^2}{2}\cdot \left|\left(\frac{6}{\sqrt{i+2}}+1\right)\cdot\frac{\hat{\beta}_{i+1}}{\sqrt{\frac{i+1}{i}}-\hat{\beta}_{i+1}}-\frac{\sqrt{\frac{1}{i}}+1}{1-\hat{\beta}_i} +1 \right| \cdot \sigma^2\right]}_{T_{3.2.1}}\\
            & +\underbrace{\mathbb{E}\left[\sum_{i=1}^t \frac{15\alpha_i^2(1+\beta_i)^2}{2}\cdot \left|\left(\frac{6}{\sqrt{i+2}}+1\right)\cdot\frac{\hat{\beta}_{i+1}}{\sqrt{\frac{i+1}{i}}-\hat{\beta}_{i+1}}-\frac{\sqrt{\frac{1}{i}}+1}{1-\hat{\beta}_i} +1 \right| \cdot \left\|\nabla f(\vtheta_i)\right\|^2\right]}_{T_{3.2.2}}
        \end{split}
    \end{equation}
    where we have $\mathcircled1$ and $\mathcircled2$ becuase of \Eqref{eq:T_2_3_constant_upper_bound} and Lemma~\ref{lem:mini_batch_norm_bound}. Then we rearrange the terms of $T_{3.2.1}$ in the following
    \begin{equation*}
        \begin{split}
            T_{3.2.1} =& \mathbb{E}\left[\sum_{i=1}^t \frac{15\alpha_0^2(1+\beta_i)^2\sigma^2}{2}\cdot \frac{1}{i} \cdot \left|\left(\frac{6}{\sqrt{i+2}}+1\right)\cdot\frac{\hat{\beta}_{i+1}}{\sqrt{\frac{i+1}{i}}-\hat{\beta}_{i+1}}-\left(\sqrt{\frac{1}{i}}+1\right)\cdot\frac{1}{1-\hat{\beta}_i} +1 \right| \right]\\
            \mathop{=}^{\mathcircled1}& \mathbb{E}\left[\sum_{i=1}^t \frac{15\alpha_0^2(1+\beta_i)^2\sigma^2}{2}\cdot \frac{1}{i} \cdot \left[\left(\sqrt{\frac{1}{i}}+1\right)\cdot\frac{1}{1-\hat{\beta}_i} - \left(\frac{6}{\sqrt{i+2}}+1\right)\cdot\frac{\hat{\beta}_{i+1}}{\sqrt{\frac{i+1}{i}}-\hat{\beta}_{i+1}}-1 \right] \right]
        \end{split}
    \end{equation*}
    where $\mathcircled1$ holds due to Corollary~\ref{cor:T_2_3_no_negative_property}. Similar to the techniques we utilized to bound $T_{2.3.1}$ of \Eqref{eq:T_2_3_refor}, we have 
    \begin{equation}
        \label{eq:T_3_2_1_refor}
        \begin{split}
            T_{3.2.1}\le &\frac{15\alpha_0^2\left(1+ \frac{1}{50}\right)^2\sigma^2}{2} \cdot\mathbb{E}\left[\frac{2}{1-\hat{\beta}_1}-1\right] + \frac{15\alpha_0^2\left(1+ \frac{1}{50}\right)^2\sigma^2}{2} \cdot  \mathbb{E}\left[\sum_{i=1}^{t-1} \frac{1}{i+1}\cdot \left(\frac{\sqrt{\frac{1}{i+1}}+1}{1-\hat{\beta}_{i+1}}-1\right) \right.\\
            & \left.-\frac{1}{i}\cdot\left(\frac{6}{\sqrt{i+2}}+1\right)\cdot\frac{\hat{\beta}_{i+1}}{\sqrt{\frac{i+1}{i}}-\hat{\beta}_{i+1}}\right]\\
            \le &15\alpha_0^2\sigma^2 \cdot\mathbb{E}\left[\frac{2}{1-\hat{\beta}_1}-1\right] + 15\alpha_0^2\sigma^2 \cdot \mathbb{E}\left[\sum_{i=1}^{t-1}\frac{1}{i}\cdot\left(\frac{\sqrt{\frac{1}{i+1}}+1}{1-\hat{\beta}_{i+1}}-1 \right.\right.\\
            & \left.\left.-\left(\frac{6}{\sqrt{i+2}}+1\right)\cdot\frac{\hat{\beta}_{i+1}}{\sqrt{\frac{i+1}{i}}-\hat{\beta}_{i+1}} \right)\right]\\
            \mathop{\le}^{\mathcircled1} & 15\alpha_0^2\sigma^2 \cdot \mathbb{E}\left[\frac{2}{1-\hat{\beta}_i} - 1\right] + 15\alpha_0^2\sigma^2 \cdot \mathbb{E}\left[\sum_{i=1}^{t-1}\frac{1}{i}\cdot \frac{6}{\sqrt{i+2}}\right]
            \mathop{\le}^{\mathcircled2}  210\alpha_0^2\sigma^2
        \end{split}
    \end{equation}
    where $\mathcircled2$ establishes becuase
    \begin{equation}
        \label{eq:int_inequality}
        \frac{2}{1-\hat{\beta}_i} \le 3,\ \forall\ i\ge 1,\quad \mathrm{and} \quad \mathbb{E}\left[\sum_{i=1}^{t-1}\frac{1}{i}\cdot \frac{1}{\sqrt{i+2}}\right] \le \int_1^t i^{-1.5} di \le 2. 
    \end{equation}
    For $T_{3.2.2}$ in \Eqref{eq:T_3_2_refor}, we then obtain
    \begin{equation}
        \label{eq:T_3_2_2_refor}
        \begin{split}
            T_{3.2.2} &= \mathbb{E}\left[\sum_{i=1}^t \frac{15\alpha_i^2(1+\beta_i)^2}{2}\cdot \left|\left(\frac{6}{\sqrt{i+2}}+1\right)\cdot\frac{\hat{\beta}_{i+1}}{\sqrt{\frac{i+1}{i}}-\hat{\beta}_{i+1}}-\frac{\sqrt{\frac{1}{i}}+1}{1-\hat{\beta}_i} +1 \right| \cdot \left\|\nabla f(\vtheta_i)\right\|^2\right]\\
            & \le \mathbb{E}\left[\sum_{i=1}^t \frac{75\alpha_i^2(1+\beta_i)^2}{2} \cdot \left\|\nabla f(\vtheta_i)\right\|^2\right].
        \end{split}
    \end{equation}
    Submitting \Eqref{eq:T_3_2_1_refor} and \Eqref{eq:T_3_2_2_refor} back into \Eqref{eq:T_3_2_refor}, we find
    \begin{equation}
        \label{eq:T_3_2_final}
        \begin{split}
            T_{3.2}\le 210\alpha_0^2\sigma^2 + \mathbb{E}\left[\sum_{i=1}^t \frac{75\alpha_i^2(1+\beta_i)^2}{2} \cdot \left\|\nabla f(\vtheta_i)\right\|^2\right].
        \end{split}
    \end{equation}
    For $T_{3.3}$ of \Eqref{eq:t3_for_z_lipschitz_continuty}, we have
    \begin{equation}
        \label{eq:T_3_3_refor}
        \begin{split}
            T_{3.3} = &\mathbb{E}\left[\sum_{i=2}^t 3\cdot\left\|\left(\left(\frac{6}{\sqrt{i+1}}+1\right)\cdot\frac{\alpha_{i-1}\hat{\beta}_i}{\sqrt{\frac{i}{i-1}}-\hat{\beta}_i} - \left(\sqrt{\frac{1}{i}}+1\right)\cdot\frac{\alpha_i\hat{\beta}_i}{1-\hat{\beta}_i}\right)\hat{m}_{i-1}\right\|^2\right]\\
            =& \mathbb{E}\left[\sum_{i=2}^t 3\cdot \left[\left(\frac{6}{\sqrt{i+1}}+1\right)\cdot\frac{\alpha_{i-1}\hat{\beta}_i}{\sqrt{\frac{i}{i-1}}-\hat{\beta}_i} - \left(\sqrt{\frac{1}{i}}+1\right)\cdot\frac{\alpha_i\hat{\beta}_i}{1-\hat{\beta}_i}\right]^2 \cdot \left\|\hat{m}_{i-1}\right\|^2\right]\\
            \mathop{\le}^{\mathcircled1} & \mathbb{E}\left[\sum_{i=2}^t 21\alpha_{i-1}\cdot \left|\left(\frac{6}{\sqrt{i+1}}+1\right)\cdot\frac{\alpha_{i-1}\hat{\beta}_i}{\sqrt{\frac{i}{i-1}}-\hat{\beta}_i} - \left(\sqrt{\frac{1}{i}}+1\right)\cdot\frac{\alpha_i\hat{\beta}_i}{1-\hat{\beta}_i}\right| \cdot \left\|\hat{m}_{i-1}\right\|^2\right]\\
            \mathop{\le}^{\mathcircled2} & \underbrace{\mathbb{E}\left[\sum_{i=2}^t 11\alpha_{i-1}\cdot \left|\left(\frac{6}{\sqrt{i+1}}+1\right)\cdot\frac{\alpha_{i-1}\hat{\beta}_i}{\sqrt{\frac{i}{i-1}}-\hat{\beta}_i} - \left(\sqrt{\frac{1}{i}}+1\right)\cdot\frac{\alpha_i\hat{\beta}_i}{1-\hat{\beta}_i}\right| \cdot \left\|\nabla f(\vtheta_{i-1})\right\|^2\right]}_{T_{3.3.1}}\\
            &+\underbrace{\mathbb{E}\left[\sum_{i=2}^t 11\alpha_{i-1}\cdot \left|\left(\frac{6}{\sqrt{i+1}}+1\right)\cdot\frac{\alpha_{i-1}\hat{\beta}_i}{\sqrt{\frac{i}{i-1}}-\hat{\beta}_i} - \left(\sqrt{\frac{1}{i}}+1\right)\cdot\frac{\alpha_i\hat{\beta}_i}{1-\hat{\beta}_i}\right| \cdot \sigma^2\right]}_{T_{3.3.2}},
        \end{split}
    \end{equation}
    where we have $\mathcircled1$ due to the fact that
    \begin{equation*}
        \begin{split}
            &\left|\left(\frac{6}{\sqrt{i+1}}+1\right)\cdot\frac{\alpha_{i-1}\hat{\beta}_i}{\sqrt{\frac{i}{i-1}}-\hat{\beta}_i} - \left(\sqrt{\frac{1}{i}}+1\right)\cdot\frac{\alpha_i\hat{\beta}_i}{1-\hat{\beta}_i}\right|\\
            \le &\left|\left(\frac{6}{\sqrt{i+1}}+1\right)\cdot\frac{\alpha_{i-1}\hat{\beta}_i}{\sqrt{\frac{i}{i-1}}-\hat{\beta}_i}\right| +\left| \left(\sqrt{\frac{1}{i}}+1\right)\cdot\frac{\alpha_i\hat{\beta}_i}{1-\hat{\beta}_i}\right| \le \left(\frac{6}{\sqrt{3}}+1\right)\cdot \alpha_{i-1} + 2\alpha_i\le 7\alpha_{i-1}
        \end{split}
    \end{equation*}
    when $i\ge 1$ and $\alpha_i = \frac{\alpha_0}{\sqrt{i}}$. Besides, $\mathcircled2$ establishes becuase of Lemma~\ref{lem:mini_batch_norm_bound}. Then, we provide the upper bound of $T_{3.3.1}$ as follows
    \begin{equation}
        \label{eq:T_3_3_1_final}
        \begin{split}
            T_{3.3.1} = &\mathbb{E}\left[\sum_{i=2}^t 11\alpha_{i-1}\cdot \left|\left(\frac{6}{\sqrt{i+1}}+1\right)\cdot\frac{\alpha_{i-1}\hat{\beta}_i}{\sqrt{\frac{i}{i-1}}-\hat{\beta}_i}- \left(\sqrt{\frac{1}{i}}+1\right)\cdot\frac{\alpha_i\hat{\beta}_i}{1-\hat{\beta}_i}\right| \cdot \left\|\nabla f(\vtheta_{i-1})\right\|^2\right]\\
            \le &\mathbb{E}\left[\sum_{i=2}^t 77\alpha_{i-1}^2\cdot \left\|\nabla f(\vtheta_{i-1})\right\|^2\right]\le \mathbb{E}\left[\sum_{i=1}^t 77\alpha_{i}^2 \cdot \left\|\nabla f(\vtheta_{i})\right\|^2\right].
        \end{split}
    \end{equation}
    After that, we find $T_{3.3.2}$ of \Eqref{eq:T_3_3_refor} satisfies
    \begin{equation}
        \label{eq:T_3_3_2_final}
        \begin{split}
            T_{3.3.2}=&\mathbb{E}\left[\sum_{i=2}^t 11\alpha_{i-1}\cdot \left|\left(\frac{6}{\sqrt{i+1}}+1\right)\cdot\frac{\alpha_{i-1}\hat{\beta}_i}{\sqrt{\frac{i}{i-1}}-\hat{\beta}_i} - \left(\sqrt{\frac{1}{i}}+1\right)\cdot\frac{\alpha_i\hat{\beta}_i}{1-\hat{\beta}_i}\right| \cdot \sigma^2\right]\\
            = & \mathbb{E}\left[11\alpha_0 \cdot \left(\frac{6}{\sqrt{3}}+1\right)\cdot \frac{\alpha_0\hat{\beta}_2}{\sqrt{2}-\hat{\beta}_2}\cdot \sigma^2\right] -\mathbb{E}\left[\left(\sqrt{\frac{1}{t}}+1\right)\cdot 11\alpha_t \cdot \frac{\alpha_t\hat{\beta}_t}{1-\hat{\beta}_t}\cdot \sigma^2\right]\\
            &+\mathbb{E}\left[\sigma^2 \sum_{i=2}^{t-1} \left(11\alpha_i \cdot \left(\frac{6}{\sqrt{i+2}}+1\right)\cdot \frac{\alpha_i\hat{\beta}_{i+1}}{\sqrt{\frac{i+1}{i}}-\hat{\beta}_{i+1}}-11\alpha_{i+1}\cdot \left(\sqrt{\frac{1}{i}}+1\right)\cdot \frac{\alpha_i\hat{\beta}_i}{1-\hat{\beta}_i}\right)\right]\\
            \le &\mathbb{E}\left[11\alpha_0^2\sigma^2\right]+ \mathbb{E}\left[\sigma^2 \sum_{i=2}^{t-1} 11\alpha_i\cdot \left( \left(\frac{6}{\sqrt{i+2}}+1\right)\cdot \frac{\alpha_i\hat{\beta}_{i+1}}{\sqrt{\frac{i+1}{i}}-\hat{\beta}_{i+1}}-\left(\sqrt{\frac{1}{i}}+1\right)\cdot \frac{\alpha_i\hat{\beta}_i}{1-\hat{\beta}_i}\right)\right]\\
            \mathop{\le}^{\mathcircled1} &\mathbb{E}\left[\sigma^2\sum_{i=1}^{t-1}11\alpha_i\cdot\frac{\alpha_0}{i}\right] \mathop{\le}^{\mathcircled2} \mathbb{E}\left[11\sigma^2\alpha_0^2 \sum_{i=1}^{t-1}i^{-1.5}\right] \mathop{\le}^{\mathcircled3} 22\sigma^2\alpha_0^2.
        \end{split}
    \end{equation}
    where $\mathcircled1$, $\mathcircled2$ and $\mathcircled3$ establish due to Lemma~\ref{lem:T_2_1_telescoping_sum_property}, $\alpha_i = \frac{\alpha_0}{\sqrt{i}}$ and \Eqref{eq:int_inequality}. Then, we introduce \Eqref{eq:T_3_3_2_final} and \Eqref{eq:T_3_3_1_final} to \Eqref{eq:T_3_3_refor}, and have
    \begin{equation}
        \label{eq:T_3_3_final}
        \begin{split}
            T_{3.3}\le 22\sigma^2\alpha_0^2 + \mathbb{E}\left[\sum_{i=1}^t 77\alpha_{i}^2 \cdot \left\|\nabla f(\vtheta_{i})\right\|^2\right].
        \end{split}
    \end{equation}
    Thus, we obtain that $T_3$ of \Eqref{eq:z_lipschitz_continuty} can be bounded as
    \begin{equation}
        \label{eq:T_3_final}
        \begin{split}
            \mathbb{E}\left[\sum_{i=1}^t L\left\|\vtheta_{i+1}^\prime - \vtheta_{i}^\prime\right\|^2\right] \le &3\sigma^2\alpha_0^2\ln(t) + 235\alpha_0^2\sigma^2+ \mathbb{E}\left[\sum_{i=1}^t 155\alpha_i^2\left\|\nabla f(\vtheta_i)\right\|^2\right].
        \end{split}
    \end{equation}
    according to \Eqref{eq:T_3_1_final}, \Eqref{eq:T_3_2_final} and \Eqref{eq:T_3_3_final}.

    Finally, if we set
    \begin{equation}
        \label{eq:alpha_0_selection}
        \begin{split}
            \alpha_0\le \frac{3}{4L+1240}
        \end{split}
    \end{equation}
    and submit \Eqref{eq:t1_for_z_lipschitz_continuty}, \Eqref{eq:t2_final} and \Eqref{eq:T_3_final} back into \Eqref{eq:z_lipschitz_continuty}, we then obtain that
    \begin{equation}
        \label{eq:final_results}
        \begin{split}
            \mathbb{E}\left[f(\vtheta^\prime_{t+1})-f(\vtheta^\prime_1)\right] \le C_0 + C_1\ln(t) - \mathbb{E}\left[\sum_{i=1}^t \frac{\alpha_i}{2} \left\|\nabla f(\vtheta_i)\right\|^2\right]
        \end{split}
    \end{equation}
    where $C_0$ and $C_1$ satisfy
    \begin{equation*}
        \begin{split}
            C_0 = \left(\frac{L}{2}+235\right)\cdot \alpha_0^2\sigma^2+\left(70G^2 +1\right)\alpha_0,\quad C_1 =\left(\frac{L}{2}+3\right)\cdot\alpha_0^2\sigma^2 + \left(60G^2 +1\right)\alpha_0.
        \end{split}
    \end{equation*} 
    As a result, we have
    \begin{equation}
        \label{eq:final_convergence}
        \begin{split}
            \frac{\mathbb{E}\left[\sum_{i=1}^t \alpha_i\left\|\nabla f(\vtheta_i)\right\|^2\right]}{\sum_{i=1}^t \alpha_i} \le \frac{\mathbb{E}\left[\sum_{i=1}^t \alpha_i\left\|\nabla f(\vtheta_i)\right\|^2\right]}{2\alpha_0\sqrt{t}} \le \frac{C_0^\prime}{\sqrt{t}} + \frac{C_1^\prime\ln(t)}{\sqrt{t}},
        \end{split}
    \end{equation}
    where
    \begin{equation}
        \label{eq:final_constant}
        \begin{split}
            C_0^\prime = \left(\frac{L}{2}+235\right)\cdot \alpha_0\sigma^2+\left(70G^2 +1\right),\quad C_1^\prime = \left(\frac{L}{2}+3\right)\cdot\alpha_0\sigma^2 + \left(60G^2 +1\right).
        \end{split}
    \end{equation}
    Thus, the proof has completed.
\end{proof}
\subsection{Some Lemmas for Theorem 4.1}
\begin{lemma}
    \label{lem:T_2_3_telescoping_sum_property}
    In Algorithm 1, if we denote
    \begin{equation*}
        \hat{\beta}_t:= \beta_t \cdot \frac{\left\|\rvg_t\right\|}{\left\|\hat{\rvm}_{t-1}\right\| + \delta_t},
    \end{equation*}
    where the parameter satisfies $\beta_t\le \frac{1}{50}$. We can obtain that
    \begin{equation*}
        \left(\frac{6}{\sqrt{i+1}}+1\right)\frac{\hat{\beta}_t}{\sqrt{\frac{i}{i-1}}-\hat{\beta}_i} + \frac{6}{\sqrt{i+1}} +1 \ge \frac{1}{\sqrt{i}}\cdot \frac{1}{1-\hat{\beta}_i} + \frac{1}{1-\hat{\beta}_i},
    \end{equation*}
    when $i\ge 2$.
\end{lemma}
\begin{proof}
    With the fact that the inequality $4i^2-9i+3 > 0$ satisfies,when $i\ge 2$, we have
    \begin{equation*}
        \left[2\left(i-1\right)\right]^2\ge i+1,\quad i\ge 2 \Longleftrightarrow \frac{\sqrt{i+1}}{i-1}\le 2,\quad i\ge 2.
    \end{equation*}
    Combining with the fact that $\hat{\beta}_i\le \frac{1}{2}$ and $\sqrt{\frac{i+1}{i}}\le 2$ when $i\ge 2$, we have 
    \begin{equation}
        \label{eq:upper_bound_with_i_selection}
        \begin{split}
        6-2\sqrt{\frac{i+1}{i}}\ge 2 \ge \frac{\sqrt{i+1}}{i-1} \ge 2\hat{\beta}_i\frac{\sqrt{i+1}}{i-1},\ i\ge 2
        \Longleftrightarrow \frac{6}{\sqrt{i+1}} - \frac{2}{\sqrt{i}}\ge \frac{2\hat{\beta}_i}{i-1} ,\ i\ge 2.
        \end{split}
    \end{equation}
    When $i\ge 2$, we then obtain that
    \begin{equation}
        \label{eq:inequality_4_mid_term_upper_bound}
        \begin{split}
            \frac{2\hat{\beta}_i}{i-1}\le \frac{6}{\sqrt{i+1}} - \frac{2}{\sqrt{i}} \mathop{\le}^{\mathcircled1} \frac{6}{\sqrt{i+1}} - \frac{1}{\sqrt{i}} \cdot \frac{1}{1-\hat{\beta}_i}\le \frac{6}{\sqrt{i+1}} \cdot \frac{\sqrt{\frac{i}{i-1}}}{\sqrt{\frac{i}{i-1}}-\hat{\beta}_t} - \frac{1}{\sqrt{i}}\cdot \frac{1}{1-\hat{\beta}_i},  
        \end{split}
    \end{equation}
    where $\mathcircled1$ establishes because of $\hat{\beta}_i\le \frac{1}{2}$. Besides, we provide some lower bound of $\frac{2\hat{\beta}_i}{i-1}$, which can be presented as
    \begin{equation}
        \label{eq:inequality_4_mid_term_lower_bound}
        \begin{split}
            & \frac{2\hat{\beta}_i}{i-1}\ge 4\cdot \left(\frac{\frac{i}{i-1} - 1}{\sqrt{\frac{i}{i-1}}+1}\right) \hat{\beta}_i\ge \frac{1}{\left(1-\hat{\beta}_i\right)^2}\cdot \left(\frac{\frac{i}{i-1} - 1}{\sqrt{\frac{i}{i-1}}+1}\right) \hat{\beta}_i\\
            = & \frac{\left(\sqrt{\frac{i}{i-1}}-1\right)\hat{\beta}_i}{\left(1-\hat{\beta}_i\right)^2}\ge \frac{\left(\sqrt{\frac{i}{i-1}}-1\right)\hat{\beta}_i}{\left(\sqrt{\frac{i}{i-1}}-\hat{\beta}_i\right)\left(1-\hat{\beta}_i\right)} = \frac{1}{1-\hat{\beta}_i} - \frac{\sqrt{\frac{i}{i-1}}}{\sqrt{\frac{i}{i-1}}-\hat{\beta}_i}
        \end{split}
    \end{equation}
    when $i\ge 2$. Combining~\Eqref{eq:inequality_4_mid_term_upper_bound} and~\Eqref{eq:inequality_4_mid_term_lower_bound}, we then obtain
    \begin{equation}
        \label{eq:inequality_4_establishment}
        \begin{split}
            & \frac{6}{\sqrt{i+1}} \cdot \frac{\sqrt{\frac{i}{i-1}}}{\sqrt{\frac{i}{i-1}}-\hat{\beta}_t} - \frac{1}{\sqrt{i}}\cdot \frac{1}{1-\hat{\beta}_i}\ge \frac{1}{1-\hat{\beta}_i} - \frac{\sqrt{\frac{i}{i-1}}}{\sqrt{\frac{i}{i-1}}-\hat{\beta}_i}\\
            \Longleftrightarrow & \frac{6}{\sqrt{i+1}} \cdot \frac{\sqrt{\frac{i}{i-1}}}{\sqrt{\frac{i}{i-1}}-\hat{\beta}_t} +\frac{\sqrt{\frac{i}{i-1}}}{\sqrt{\frac{i}{i-1}}-\hat{\beta}_i}- \frac{1}{\sqrt{i}}\cdot \frac{1}{1-\hat{\beta}_i} - \frac{1}{1-\hat{\beta}_i} \ge 0\\
            \Longleftrightarrow & \left(\frac{6}{\sqrt{i+1}}+1\right)\frac{\hat{\beta}_t}{\sqrt{\frac{i}{i-1}}-\hat{\beta}_i} + \frac{6}{\sqrt{i+1}} +1 \ge \frac{1}{\sqrt{i}}\cdot \frac{1}{1-\hat{\beta}_i} + \frac{1}{1-\hat{\beta}_i}.
        \end{split}
    \end{equation}
    to complete the proof.
\end{proof}
\begin{corollary}
    \label{cor:T_2_1_no_negative_property}
    Under the hypotheses of Lemma~\ref{lem:T_2_3_telescoping_sum_property}, we have
    \begin{equation*}
        \frac{1}{\sqrt{i-1}} \cdot \left[\left(\frac{6}{\sqrt{i+1}}+1\right)\cdot \frac{\hat{\beta}_i}{\sqrt{\frac{i}{i-1}}-\hat{\beta}_i} \right] \ge \frac{1}{\sqrt{i}} \cdot\left(\frac{\left(\frac{1}{\sqrt{i}}+1\right)\hat{\beta}_i}{1-\hat{\beta}_i}\right)
    \end{equation*}
\end{corollary}
\begin{proof}
    Acoording to~Lemma~\ref{lem:T_2_3_telescoping_sum_property}, the LHS of~\Eqref{eq:inequality_4_establishment}
    can be reformulated as 
    \begin{equation}
        \label{eq:inequality_3prime}
        \begin{split}
            \left(\frac{6}{\sqrt{i+1}}+1\right)\frac{\hat{\beta}_t}{\sqrt{\frac{i}{i-1}}-\hat{\beta}_i} + \frac{6}{\sqrt{i+1}} +1 = \left(\frac{6}{\sqrt{i+1}}+1\right)\cdot \frac{\hat{\beta}_i}{\sqrt{\frac{i}{i-1}}-\hat{\beta}_i} \cdot \frac{\sqrt{\frac{i}{i-1}}}{\hat{\beta}_i}
        \end{split}
    \end{equation}
    Thus, submitting~\Eqref{eq:inequality_3prime} back to~\Eqref{eq:inequality_4_establishment}, we have
    \begin{equation}
        \label{eq:inequality_final_1}
        \frac{1}{\sqrt{i-1}} \cdot \left[\left(\frac{6}{\sqrt{i+1}}+1\right)\cdot \frac{\hat{\beta}_i}{\sqrt{\frac{i}{i-1}}-\hat{\beta}_i} \right] \ge \frac{1}{\sqrt{i}} \cdot\left(\frac{\left(\frac{1}{\sqrt{i}}+1\right)\hat{\beta}_i}{1-\hat{\beta}_i}\right)
    \end{equation} 
    to complete the proof.
\end{proof}
\begin{lemma}
    \label{lem:T_2_1_telescoping_sum_property}
    Under the hypotheses of Lemma~\ref{lem:T_2_3_telescoping_sum_property}, we have
    \begin{equation*}
        \frac{\hat{\beta}_{i+1}}{\sqrt{\frac{i+1}{i}}-\hat{\beta}_{i+1}}\cdot \left(\frac{6}{\sqrt{i+2}}+1\right) - \frac{\hat{\beta}_i}{1-\hat{\beta}_i}\left(\frac{1}{\sqrt{i}}+1\right) \le \frac{1}{\sqrt{i}}
    \end{equation*}
\end{lemma}
\begin{proof}
    According to the update paradigm of Algorithm 1, we have
    \begin{equation}
        \label{eq:inequality_2prime_sufficient_condition_1}
        \frac{\hat{\beta}_{i+1}}{\hat{\beta}_i}\le \sqrt{\frac{i+1}{i}} \Leftrightarrow \frac{1}{\hat{\beta}_i} \le \frac{\sqrt{\frac{i+1}{i}}}{\hat{\beta}_{i+1}}\Leftrightarrow \frac{1-\hat{\beta}_i}{\hat{\beta}_i}\le \frac{\sqrt{\frac{i+1}{i}}-\hat{\beta}_{i+1}}{\hat{\beta}_{i+1}}\Leftrightarrow \frac{\hat{\beta}_i}{1-\hat{\beta}_i}\ge \frac{\hat{\beta}_{i+1}}{\sqrt{\frac{i+1}{i}}-\hat{\beta}_{i+1}}.
    \end{equation}
    Besides, with the fact that 
    \begin{equation*}
        \begin{split}
            \min_{j\ge 2}\ \frac{\sqrt{(j+1)(j+2)}}{j} - \frac{1}{2}\cdot \sqrt{\frac{j+2}{j}} \ge \frac{1}{2}
        \end{split}
    \end{equation*}
    and $\hat{\beta}_i\le \frac{1}{12}$ for all $i\ge 2$, we then obtain
    \begin{equation}
        \label{eq:inequality_2prime_sufficient_condition_2}
        \begin{split}
            & 6\hat{\beta}_{i+1}\le \frac{1}{2}\le \left(\frac{\sqrt{(i+1)(i+2)}}{i} - \frac{1}{2}\cdot \sqrt{\frac{i+2}{i}}\right)\\
            = & \left(\sqrt{\frac{i+1}{i}}-\frac{1}{2}\right)\sqrt{\frac{i+2}{i}} \le \frac{\sqrt{\frac{i+1}{i}}-\hat{\beta}_{i+1}}{1-\hat{\beta}_i}\cdot\sqrt{\frac{i+2}{i}}.
        \end{split}
    \end{equation}
    We then rearrange the terms of~\Eqref{eq:inequality_2prime_sufficient_condition_2}, the following inequality establishes
    \begin{equation}
        \label{eq:inequality_2prime_sufficient_condition_2_re}
        \frac{1}{1-\hat{\beta}_{i}}\cdot \frac{1}{\sqrt{i}} \ge \frac{\hat{\beta}_{i+1}}{\sqrt{\frac{i+1}{i}}-\hat{\beta}_{i+1}}\cdot \frac{6}{\sqrt{i+2}}
    \end{equation}
    Combining~\Eqref{eq:inequality_2prime_sufficient_condition_1} with~\Eqref{eq:inequality_2prime_sufficient_condition_2_re}, we have
    \begin{equation}
        \label{eq:inequality_2prime_sum}
        \begin{split}
            & \frac{\hat{\beta}_i}{1-\hat{\beta}_i}+\frac{1}{1-\hat{\beta}_{i}}\cdot \frac{1}{\sqrt{i}} \ge \frac{\hat{\beta}_{i+1}}{\sqrt{\frac{i+1}{i}}-\hat{\beta}_{i+1}}\cdot \left(\frac{6}{\sqrt{i+2}}+1\right)\\
            \Longleftrightarrow & \frac{\hat{\beta}_i}{1-\hat{\beta}_i}\left(\frac{1}{\sqrt{i}}+1\right)+\frac{1}{\sqrt{i}} \ge \frac{\hat{\beta}_{i+1}}{\sqrt{\frac{i+1}{i}}-\hat{\beta}_{i+1}}\cdot \left(\frac{6}{\sqrt{i+2}}+1\right)\\
            \Longleftrightarrow & \frac{\hat{\beta}_{i+1}}{\sqrt{\frac{i+1}{i}}-\hat{\beta}_{i+1}}\cdot \left(\frac{6}{\sqrt{i+2}}+1\right) - \frac{\hat{\beta}_i}{1-\hat{\beta}_i}\left(\frac{1}{\sqrt{i}}+1\right) \le \frac{1}{\sqrt{i}}.
        \end{split}
    \end{equation}
    Thus, we complete the proof.
\end{proof}
\begin{corollary}
    \label{cor:T_2_3_no_negative_property}
    Under the hypotheses of Lemma~\ref{lem:T_2_3_telescoping_sum_property}, we have
    \begin{equation*}
        \frac{\hat{\beta}_{i+1}}{\sqrt{\frac{i+1}{i}}-\hat{\beta}_{i+1}}\cdot \left(\frac{6}{\sqrt{i+2}}+1\right) - \frac{\sqrt{\frac{1}{i}}+1}{1-\hat{\beta}_i} +1 \le 0
    \end{equation*}
\end{corollary}
\begin{proof}
    According to Lemma~\ref{lem:T_2_1_telescoping_sum_property}, we have
    \begin{equation}
        \label{eq:telescoping_sum_property_results}
        \frac{\hat{\beta}_{i+1}}{\sqrt{\frac{i+1}{i}}-\hat{\beta}_{i+1}}\cdot \left(\frac{6}{\sqrt{i+2}}+1\right) - \frac{\hat{\beta}_i}{1-\hat{\beta}_i}\left(\frac{1}{\sqrt{i}}+1\right) \le \frac{1}{\sqrt{i}} = \left(1-\hat{\beta}_i\right)\cdot \frac{\sqrt{\frac{1}{i}}+1}{1-\hat{\beta}_i}-1.
    \end{equation}
    Thus, rearrange the previous inequalities, we obtain
    \begin{equation}
        \label{eq:T_2_3_no_negative_results}
        \begin{split}
            \frac{\hat{\beta}_{i+1}}{\sqrt{\frac{i+1}{i}}-\hat{\beta}_{i+1}}\cdot \left(\frac{6}{\sqrt{i+2}}+1\right) - \frac{\sqrt{\frac{1}{i}}+1}{1-\hat{\beta}_i} +1 \le 0
        \end{split}
    \end{equation}
    to complete the proof.
\end{proof}

\end{document}